\documentclass[oneside,english]{amsart}
\usepackage[T1]{fontenc}
\usepackage[latin9]{inputenc}
\usepackage{mathrsfs}
\usepackage{amsthm}
\usepackage{amssymb}
\usepackage[all]{xy}

\makeatletter
\numberwithin{equation}{section}
\numberwithin{figure}{section}
\theoremstyle{plain}
\newtheorem{thm}{\protect\theoremname}[section]
  \theoremstyle{remark}
  \newtheorem{rem}[thm]{\protect\remarkname}
  \theoremstyle{definition}
  \newtheorem{defn}[thm]{\protect\definitionname}
  \theoremstyle{plain}
  \newtheorem{lem}[thm]{\protect\lemmaname}
  \theoremstyle{plain}
  \newtheorem{cor}[thm]{\protect\corollaryname}
  \theoremstyle{plain}
  \newtheorem{prop}[thm]{\protect\propositionname}
  \theoremstyle{definition}
  \newtheorem{example}[thm]{\protect\examplename}
  \theoremstyle{remark}
  \newtheorem*{rem*}{\protect\remarkname}

\makeatother

\usepackage{babel}
  \providecommand{\corollaryname}{Corollary}
  \providecommand{\definitionname}{Definition}
  \providecommand{\examplename}{Example}
  \providecommand{\lemmaname}{Lemma}
  \providecommand{\propositionname}{Proposition}
  \providecommand{\remarkname}{Remark}
\providecommand{\theoremname}{Theorem}

\begin{document}

\title{metric $1$-spaces}

\author{Ittay Weiss}
\begin{abstract}
A generalization of metric space is presented which is shown to admit
a theory strongly related to that of ordinary metric spaces. To avoid
the topological effects related to droping any of the axioms of metric
space, first a new, and equivalent, axiomatization of metric space
is given which is then generalized from a fresh point of view. Naturally
arising examples from metric geometry are presented. 
\end{abstract}

\maketitle

\section{\label{sec:Introduction}Introduction}

Symmetry, in the axiomatization of metric space, was under stress
early on in the development of the theory, for various reasons. Simply
neglecting the axiom of symmetry carries with it certain topological
difficulties that are discomforting, thus dictating proceeding with
caution. Below we argue, by presenting a completely equivalent axiomatization
of metric space that makes no mention of symmetry, that symmetry as
an axiom had received far too much attention. This new axiomatization
enables us to develop a new extension of metric space whose study
is the aim of this work. But first, some more background information
and motivation regarding metric spaces, the symmetry axiom, and its
weakening.

\subsection*{History}

The idea that the abstract notion of distance is a symmetric one appears
to be deeply ingrained and probably emanates from our basic intuition
about distance in Euclidean spaces. However, the subtleties of the
imposition of symmetry were noted early on in the development of metric
spaces and General Topology. In 1918 Finsler introduced in \cite{finsler1918ueber},
his Ph.D. dissertation, what was later coined by Élie Cartan \emph{Finsler
geometry}; a generalization of Riemannian manifold having an induced
quasimetric instead of a metric. In 1931 Wilson, in an article titled
{}``On quasi metric spaces'' (\cite{wilson1931quasi}), comments
that: 
\begin{quotation}
{}``In one sense a quasi-metric space is merely the result of suppressing
the axiom {[}of symmetry{]} from the definition of metric space. Usually
the result of such a limitation on a set of axioms is to diminish
the number of theorems easily deducible, but in this case there is
an embarrassing richness of material''.
\end{quotation}
Some 38 years later, Stoltenberg, in \cite{stoltenberg1969quasi},
an article bearing the exact same title, reviews the than recent developments
in the study of the implications of removing the limitation of symmetry.
More on the history of metric spaces can be found in \cite{ribeiro1943espacesa}.

\subsection*{Topological difficulties}

The aim of this work is to develop a new generalization of metric
space with little ill effect on the accompanying topological notions.
Let us first recount the topological pathologies that arise when each
of the axioms of metric space is neglected.

\subsubsection*{Zero distance and non-Hausdorff spaces}

Allowing distinct points to have distance equal to $0$ corresponds
to the topological phenomenon of a non-Hausdorff topology. These were
initially considered to be invalid spaces, but the rich structure
of finite topologies shows such restrictions were not in place (see
\cite{MR1173270} for more on the applicability of non-Hausdorff spaces).

\subsubsection*{Positive self distance and ghostly neighborhoods}

Allowing a point to have positive distance from itself (such spaces
are called partial metrics or dislocated metrics) has the topological
effect of a point possessing neighborhoods not containing that point.
Topologically, such ghostly neighborhoods appear to be very pathological.

\subsubsection*{The triangle inequality and open balls}

It is easily seen that the triangle inequality assures that open balls
are open sets. Thus neglecting the triangle inequality will destroy
the most fundamental example of an open set.

\subsubsection*{Symmetry and closed balls}

It is equally easy to see that symmetry assures that closed balls
are closed sets. Thus, neglecting the axiom of symmetry will destroy
the most fundamental example of a closed set. 
\begin{rem}
Viewed this way, it is a bit peculiar that the triangle inequality
axiom is related to open balls being open sets while the symmetry
axiom is related to closed balls being closed sets. The two axioms
do not appear to be dual to each other and one may wonder if there
is an equivalent axiomatization that better reflects the open/closed
duality nature of topology. Indeed, an affirmative answer is provided
incidentally by the axiomatization we present below. 
\end{rem}

\subsection*{Further motivation}

The relevance of quasi-metrics to Physics and Biology goes back, respectively,
at least to 1941 with Rander's \cite{randers1941asymmetrical} and
to 1976 with Waterman, Smith, and Beyer's \cite{MR0408876}. In \cite{lawvere1973metric}
Lawvere shows that a significant portion of metric geometry that does
not rely on symmetry can be seen as a form of extended logic and,
at the same time, as a special case of the theory of enriched categories.
More recently, Vickers continued this line of investigation in \cite{vickers2005localic}
and a unifying approach is given in \cite{MR2044145}. Some general
properties of quasimetrics are presented in \cite{deza2000quasi}
and \cite{shore1996metrizability}. An asymmetric Arzelà-Ascoli theorem
is established in \cite{MR2328014}. Generalized metrics in the theory
of computation can be found in \cite{MR2783140} and \cite{seda2010generalized},
and in computation semantics in \cite{MR2013571} and \cite{MR1475088}.
For further applications in Computer Science we mention \cite{ali2009space}
and \cite{brattka2003recursive}. The recent encyclopedia of distances
\cite{deza2009encyclopedia}, including a wealth of generalizations
of metric space, is yet another reflection of the importance of these
structures. 

It thus appears that the demand for symmetry to hold in any metric
space might have been adopted prematurely. Aside from these motivating
forces for considering non-symmetry, a very naive, yet illuminating,
reason is given by Gromov in \cite{gromov2006metric} when referring
to how the symmetry axiom in the definition of a metric space unpleasantly
limits many applications:
\begin{quote}
{}``the effort of climbing up to the top of a mountain, in real life
as well as in mathematics, is not at all the same as descending back
to the starting point''.
\end{quote}
Another reason to consider various ways in which symmetry can be weakened
is of a more theoretical nature. In order to understand the role of
the axiom of symmetry in the general theory of metric spaces one is
led to study the effects of weakening it. Perhaps the most studied
instance is that of simply neglecting the axiom of symmetry. The accompanying
topological theory is that of bitopological spaces introduced in \cite{MR0143169}
with its rich structure expounded in \cite{MR2126263}. 

However, as alluded to above, the accompanying bitopological theory
does not seem to have the same intimate relation that topology has
with metric spaces. Below, we present the theory of metric $1$-spaces
for which we are able to establish fundamental results so that both
results and proofs echo the familiar theory of ordinary metric spaces
quite strongly.

\subsection*{Plan of the paper}

Section \ref{sec:Reformulating-the-classical}, after a very simple
analysis of the axioms of metric space, presents a new axiomatization
of metric space which does not mention symmetry at all. The importance
of this reformulation is that it enables new possibilities for generalizations.
Section \ref{sec:begining section weakening} then develops such a
generalization and introduces the main objects of study, namely metric
$1$-spaces. With most of the article devoted to the fine structure
of metric $1$-spaces, Section \ref{sec:The-coarse-structure} is
a detour providing a brief study of the coarse structure of metric
$1$-spaces. In it, \emph{coarse $1$-spaces }are defined and a metrizability
theorem is proven. Section  \ref{sec:convergence} is concerned with
notions of convergence in metric $1$-spaces. Forward sequences and
series are defined, as well as their dual notions of backward sequences
and series. A notion of limit for each of these entities is defined
and is shown to extend the usual notion of limits of sequences in
ordinary metric spaces. Section \ref{sec:Continuity} then introduces
appropriate notions of continuity and studies their interrelations.
Section \ref{sec:Fundamental-results} pieces together the results
of the previous sections to establish three fundamental results about
metric $1$-spaces. These results generalize the familiar results
of ordinary metric space theory regarding metrizability of function
spaces, the equivalence between continuity and uniform continuity
on compact domains, and the Banach fixed point theorem for a contracting
self-map. In Section \ref{sec:Dagger-structures-and} a hierarchy
of symmetry is introduced by means of dagger structures with different
properties. Finally, Section \ref{sec:Examples,-related-structures,}
provides more examples of metric $1$-spaces, presents several possible
applications in physics, and discusses metric $n$-spaces and their
relation to other structures.

\subsection*{Acknowledgements}

The author would like to thank Gavin Seal for an encouraging encounter
in Lausanne.

\section{\label{sec:Reformulating-the-classical}Reformulating the classical
definition}

The definition of a metric space as formulated by Fréchet in 1906
is, of course, very well known, and yet we revisit it here briefly
and in full detail. Our aim in this short section is to examine the
axiomatization and arrive at an equivalent one that makes no mention
of symmetry. While quite a simple result the author, surprisingly,
could not find any trace of it in more than several books and articles.
The significance of this reformulation of the axioms lies in the shift
of focus it enables when generalizing the notion. Suddenly, there
is no need to weaken symmetry since it is no longer demanded explicitly.
The new axiomatization suggests a new possibility for extending the
theory of metric spaces whose exploration is the aim of the sections
that follow. 

Consider a set $X$ and a function $w:X\times X\to\mathbb{R}_{+}$,
where $\mathbb{R}_{+}=\{x\in\mathbb{R}\mid x\ge0\}\cup\{\infty\}$
is the set of extended non-negative real numbers. We say that: 
\begin{itemize}
\item $w$ is \emph{locally finite} if for every $x,y\in X$ holds that
$w(x,y)\ne\infty$.
\item $w$ is \emph{non-degenerate }if, for every $x,y\in X$, $w(x,y)=0$
implies $x=y$.
\item $w$ is \emph{reflexive} if for all $x\in X$ holds that $w(x,x)=0$. 
\item $w$ is \emph{symmetric }if for all $x,y\in X$ holds that $w(x,y)=w(y,x)$. 
\item $w$ satisfies the \emph{full triangle inequality }if the inequalities
\[
|w(x,y)-w(y,z)|\le w(x,z)\le w(x,y)+w(y,z)
\]
hold for all $x,y,z\in X$. 
\item $w$ satisfies the \emph{restricted triangle inequality} if the inequality
\[
w(x,z)\le w(x,y)+w(y,z)
\]
holds for all $x,y,z\in X$.\end{itemize}
\begin{rem}
\label{Remark:convention infty}For the full and restricted triangle
inequalities we agree that $x+\infty=\infty+x=\infty+\infty=\infty$,
that $x-\infty=-\infty$, and that $\infty-x=\infty$ hold for every
real $x\ge0$ (here $-\infty$ is a new symbol satisfying $|-\infty|=\infty$).
The quantity $\infty-\infty$ is left undefined and if $w(x,y)=\infty=w(y,z)$
then the full triangle inequality is to be interpreted as setting
no particular lower bound on $w(x,z)$. We remark as well that while
technically the full triangle inequality consists of two inequalities
we still refer to it in the singular as a single entity to stress
the point that it should be treated as one whole. 
\end{rem}
With this terminology in place we may now state the common-place definition
of metric space.
\begin{defn}
A \emph{metric space} is a set $X$ together with a locally finite,
reflexive, non-degenerate and symmetric function $w:X\times X\to\mathbb{R}_{+}$
which satisfies the restricted triangle inequality. \end{defn}
\begin{lem}
\label{lem:rest plus ref is full}Let $X$ be a set and $w:X\times X\to\mathbb{R}_{+}$
a function. If $w$ is reflexive and satisfies the restricted triangle
inequality then the following are equivalent\end{lem}
\begin{itemize}
\item $w$ is symmetric.
\item $w$ satisfies the full triangle inequality.\end{itemize}
\begin{proof}
Assuming symmetry we need to establish that $|w(x,y)-w(y,z)|\le w(x,z)$
which, without loss of generality, would follow from showing that
$w(x,y)\le w(x,z)+w(y,z)$. Using symmetry this is equivalent to $w(x,y)\le w(x,z)+w(z,y)$
which holds by the restricted triangle inequality. In the other direction,
the full triangle inequality implies that $|w(x,y)-w(y,x)|\le w(x,x)=0$,
and symmetry follows. \end{proof}
\begin{cor}
\label{cor:(New,-and-Equivalent,}(Equivalent Definition of Metric
Space) A metric space is a set $X$ together with a locally finite,
reflexive and non-degenerate function $w:X\times X\to\mathbb{R}_{+}$
which satisfies the full triangle inequality. \end{cor}
\begin{rem}
Our choice of using $w:X\times X\to\mathbb{R}_{+}$ instead of the
more common $d:X\times X\to\mathbb{R}_{+}$ is meant to distantiate
the exposition from the deeply grained prejudices suggested by the
word 'distance'. The reader may thus think instead of the word 'weight'. 
\end{rem}
With this result we achieved our preliminary goal of a symmetry-free
axiomatization of metric spaces. As a by product notice that, as alluded
to in Section \ref{sec:Introduction}, this axiomatization reflects
the open/closed duality of topology in a rather straightforward manner:
Each half of the full triangle inequality corresponds to either open
balls being open sets or closed balls being closed sets. 
\begin{rem}
\label{Rem:As-Michael-Lockyer}As Michael Lockyer pointed out to the
author, the axiom of symmetry can also be replaced, in the presence
of reflexivity, by the strong triangle inequality: $d(x,z)\le d(x,y)+d(z,y)$.
Below we present a generalization of metric spaces based on the axiomatization
in Corollary \ref{cor:(New,-and-Equivalent,}. It is also possible
to develop such a generalization based on the strong triangle inequality.
Such a generalization is already captured by the work below, as explained
in Remark \ref{Rem: Michael} below. 
\end{rem}

\section{\label{sec:begining section weakening}Weakening the new definition\label{sec:Weakening-the-new}}

The main point of the previous section was that we arrived at an axiomatization
of classical metric spaces which makes no mention of symmetry. We
are now free to seek out a generalization from a fresh point of view
on a very old subject.

Disposing of local finiteness is most easily justified. The immediate
benefit of doing so is the existence of the coproduct of two metric
spaces (and in fact any small colimit). In fact, the local finiteness
axiom is already starting to disappear from textbooks on metric spaces,
(see, e.g., the definition of metric space in \cite{MR1835418} and
the discussion following it). 

Disposing of non-degeneracy is also easy to digest. In the literature
the resulting structure is called a semi-metric space or a quasi-metric
space. We prefer the more descriptive use of the term 'degenerate'.
We refer again to page 2 of \cite{MR1835418} for more details. In
light of these facts we will use the term 'metric space' for spaces
that might be degenerate or not locally finite. 

We are thus left with reflexivity and the full triangle inequality
and we feel most unwilling to depart with any of these due to the
topological consequences of doing so. Reflexivity ensures that every
neighborhood of a point contains that point, and the full triangle
inequality guarantees that open balls are open sets and that closed
balls are closed sets. Interestingly, though not the path we take
below, it is useful to consider relaxing these axioms (see e.g., \cite{hitzler2000dislocated}
for relevance to electronic engineering, \cite{hitzler2001generalized}
for applications in programming semantics, \cite{abramsky1994domain}
in domain theory, and \cite{MR2572106} for uses in theoretical computer
science). 

We are interested in developing a theory that retains as much of the
character of topology as possible and thus insist on retaining both
reflexivity and the full triangle inequality. As we saw above, these
two axioms imply symmetry and thus there seems to be no escape from
symmetry. However, we now identify a hidden axiom in the axiomatization,
the weakening of which will allow us to proceed.

The axiom of univalence is the assumption that for every two points
$x,y\in X$ there is a unique associated number $w(x,y)$. This assumption
stems from the idea that $w(x,y)$ should signify the distance from
$x$ to $y$ and that this distance is uniquely determined by the
end points alone. However, when measuring a quantity from $x$ to
$y$ one may consider measuring different aspects as signified by
a parameter. It thus becomes natural to refine $w$ and allow it to
become multivalued. In fact it is sensible to allow $w(x,y)$ to attain
no value at all, for instance if it is not at all possible to measure
anything from $x$ to $y$. If a measurement is possible then there
could be a whole array of parameters indicating how one should measure.
Thus, we wish to replace the underlying set of a metric space by a
richer structure known as a category. 

A category (definition follows) can be seen as a $1$-dimensional
analogue of a set. In more detail, a category is a set (or a class)
of objects, thought of as $0$-dimensional point-like elements, together
with $1$-dimensional arrows between objects together with a notion
of composition of such arrows. 
\begin{defn}
A \emph{category} $\mathscr{C}$ consists of a class of objects $ob(\mathscr{C})$
and, for every two objects $x,y\in ob(\mathscr{C})$, a set $\mathscr{C}(x,y)$.
These sets are to be  disjoint, in the sense that if $\mathscr{C}(x,y)\cap\mathscr{C}(x',y')\ne\emptyset$
then $x=x'$ and $y=y'$. An element $\psi\in\mathscr{C}(x,y)$ is
also denoted by $\psi:x\to y$ and called an \emph{arrow }or a \emph{morphism}.
The object $x$ is the \emph{domain }of $\psi$ and the object $y$
is the \emph{codomain }of $\psi$\emph{. }For each object $x\in ob(\mathscr{C})$
there is a designated arrow $id_{x}:x\to x$ called the \emph{identity
}arrow at $x$.\emph{ }Lastly, there is a composition rule that assigns
to arrows $\xymatrix{x\ar[r]^{\psi} & y\ar[r]^{\varphi} & z}
$ their \emph{composition }$\varphi\circ\psi:x\to z$. With respect
to the composition, the identity arrows are required to be \emph{neutral}
in the sense that if $\psi:x\to y$ is any arrow then $\psi\circ id_{x}=\psi$
and $id_{y}\circ\psi=\psi$. The final condition is that the composition
be \emph{associative} in the sense that if $\xymatrix{x\ar[r]^{\psi} & y\ar[r]^{\varphi} & z\ar[r]^{\rho} & w}
$ are any three arrows then $\rho\circ(\varphi\circ\psi)=(\rho\circ\varphi)\circ\psi$.
The class of all arrows in the category $\mathscr{C}$ is denoted
by $Arr(\mathscr{C})$.\end{defn}
\begin{rem}
When considering general categories, size issues can become important.
Namely, if the class of arrows in a category is a proper class then
certain constructions are not guaranteed to exist. We adopt here the
common solution due to Grothendieck of assuming implicitly a hierarchy
of universes so that the class of arrows of a given category is small
with respect to some universe. From this point on we tacitly ignore
all size issues. 
\end{rem}
Examples of categories include the category $\mathbf{Set}$ with objects
all sets and arrows all functions, the category $\mathbf{Top}$ of
all topological spaces as objects and all continuous mappings as arrows,
the category $\mathbf{Grp}$ of all groups as objects and group homomorphisms
as arrows and so on. Another class of examples of categories useful
in what follows is the following one. Any set $S$ naturally gives
rise to a category $I_{S}$, called an \emph{indiscrete category},
where
\begin{itemize}
\item $ob(I_{S})=S$
\item $I_{S}(x,y)=\{\psi_{x,y}\}$
\end{itemize}
with $id_{x}=\psi_{x,x}$ and compositions determined uniquely. Via
this construction categories can be seen to extend sets. 

Structure preserving maps between categories are known as functors.
The formal definition is as follows. 
\begin{defn}
Let $\mathscr{C}$ and $\mathscr{D}$ be categories. A \emph{functor
}$F:\mathscr{C}\to\mathscr{D}$ consists of an assignment of an object
$F(c)\in ob(\mathscr{D})$ to any object $c\in ob(\mathscr{C})$ and
to every pair $c,c'\in ob(\mathscr{C})$, a function $F:\mathscr{C}(c,c')\to\mathscr{D}(F(c),F(c'))$
such that for every $c\in ob(\mathscr{C})$ holds that $F(id_{c})=id_{F(c)}$
and for every composable pair of arrows holds that $F(\varphi\circ\psi)=F(\varphi)\circ F(\psi)$.\end{defn}
\begin{rem}
Categories were introduced in \cite{MR0013131} by Eilenberg and Mac
Lane in 1945 not as generalizations of sets but rather in order to
make precise the illusive exact meaning of the naturality of certain
mathematical constructions (such as the natural isomorphism between
a finite dimensional vector space and its double dual), an effort
that proved crucial in advancing homology theory. Later, category
theory, found uses in algebraic geometry, computer science, and logic
just to mention a few areas. For more information on categories the
reader is referred to \cite{MR1712872}.
\end{rem}
We can now formulate the final step in the weakening of the axioms
of a metric space by removing the assumption of univalence.
\begin{defn}
A \emph{metric $1$-space} is a category $X$ together with, for every
two objects $x,y\in ob(X)$, a function $w:X(x,y)\to\mathbb{R}_{+}$
which satisfies \emph{reflexivity} and the \emph{full triangle inequality}
in the following sense. \end{defn}
\begin{itemize}
\item For every $x\in ob(X)$ the equality $w(id_{x})=0$ holds.
\item For every $x,y,z\in ob(X)$ and arrows $\psi:x\to y$ and $\varphi:y\to z$
the inequalities 
\[
|w(\varphi)-w(\psi)|\le w(\varphi\circ\psi)\le w(\varphi)+w(\psi)
\]
hold.
\end{itemize}
It is assumed that we follow the same convention set out in Remark
\ref{Remark:convention infty} above regarding computations involving
$\infty$. In particular, if $w(\varphi)=w(\psi)=\infty$ then the
full triangle inequality sets no lower bound on $w(\varphi\circ\psi)$.

We note immediately that every ordinary metric space $(S,d)$ gives
rise to a metric $1$-space structure on the indiscrete category $I_{S}$
by defining $w(\psi_{x,y})=d(x,y)$, for every arrow $\psi_{x,y}$
in $I_{S}$. Moreover, any metric structure on $I_{S}$ arises in
this way as we now show.
\begin{prop}
Let $X$ be a metric $1$-space and $\psi:x\to y$ and $\varphi:y\to z$
arrows in $X$. If $w(\varphi\circ\psi)=0$ then $w(\psi)=w(\varphi)$.\end{prop}
\begin{proof}
$|w(\varphi)-w(\psi)|\le w(\varphi\circ\psi)=0$.\end{proof}
\begin{cor}
\label{cor:wf is wfinv}If $\psi$ is an isomorphism (i.e., $\psi$
has an inverse) in a metric $1$-space then $w(\psi)=w(\psi^{-1})$.
\end{cor}
Since in an indiscrete category $I_{S}$ every arrow is an isomorphism
we obtain
\begin{cor}
\label{cor:if ind then symmetric}If $X$ is a metric $1$-space with
an indiscrete underlying category $I_{S}$ then defining $d(x,y)=w(\psi_{x,y})$
defines a symmetric metric structure on $S$.
\end{cor}
We thus see that metric spaces can be identified with metric $1$-spaces
having an indiscrete underlying category. 
\begin{example}
\label{Exam:BiLip}In the context of ordinary metric spaces recall
that a function $f:X\to Y$ between metric spaces is called bi-Lipschitz
if there exists a constant $C$, called a bi-Lipschitz constant, such
that 
\[
\frac{1}{C}d(x,x')\le d(f(x),f(x'))\le Cd(x,x')
\]
holds for all $x,x'\in X$ (note that such a $C$, if it exists, satisfies
$C\ge1$). Consider now the category $\mathbf{BiLip}$ of all metric
spaces and bi-Lipschitz mappings between them. For each arrow $f:X\to Y$
in that category let $C_{f}$ be the infimum over all bi-Lipschitz
constants for $f$ and let $w(f)=\ln(C_{f})$. It is straightforward
to verify that this choice of $w$ turns $\mathbf{BiLip}$ into a
metric $1$-space. 
\end{example}
Recall, from \cite{lawvere1973metric}, that a Lawvere space is a
set $X$ equipped with a reflexive function $d:X\times X\to\mathbb{R}_{+}$
satisfying the restricted triangle inequality. Given any metric $1$-space
$(X,w)$ one may define a Lawvere structure on $S=ob(X)$ by the formula
\[
d(x,y)=\inf_{\psi:x\to y}w(\psi)
\]
for any two objects $x,y\in ob(X)$. We denote this Lawvere space
by $L(X)$ and note that it would usually fail to be a metric space
since symmetry would not generally hold.
\begin{example}
Continuing Example \ref{Exam:BiLip}, it is easily seen that in $L(\mathbf{BiLip})$
the distance $d(X,Y)$ is the usual Lipschitz distance between $X$
and $Y$. 
\end{example}
Before embarking on the study of metric $1$-spaces we close this
section by mentioning the concept of categorical duality. Given a
category $\mathscr{C}$ one may construct a new category, denoted
$\mathscr{C}^{op}$, and called the \emph{opposite category}, by formally
reversing the directions of all arrows in $\mathscr{C}$. More concretely,
$ob(\mathscr{C}^{op})=ob(\mathscr{C})$ and for every arrow $\psi:c\to d$
in $\mathscr{C}$ there is an arrow $\psi^{op}:d\to c$ in $\mathscr{C}^{op}$.
It follows that every result about categories can be dualized to give
another true result. This observation will be used repeatedly below
and we refer the reader to \cite{MR1712872} for more details on duality.
We do comment that for most categories $\mathscr{C}$ the opposite
category $\mathscr{C}^{op}$ is very different than the original category.
For instance, the opposite of the category $\mathbf{Set}$ of sets
and functions is essentially the same as the category of complete
atomic boolean algebras. Duality can also be used to define new objects
of study. For instance, non-commutative geometry \emph{defines }its
objects of study to be the objects in the opposite of a category of
algebras.

\section{\label{sec:The-coarse-structure}The coarse structure of a metric
$1$-space}

The main aim of this work is to investigate the fine structure of
a metric $1$-space. In this short section we take a detour to consider
the coarse structure of metric $1$-spaces as well. The further study
of the coarse structure of metric $1$-spaces is postponed to a future
article. 

If $E,E_{1},E_{2}$ are relations on a fixed set $X$ then recall
that $E_{1}\circ E_{2}=\{(x,z)\in X\times X\mid\exists y\in X\,\quad\,(x,y)\in E_{1},(y,z)\in E_{2}\}$
and that $E^{-1}=\{(y,x)\in X\times X\mid(x,y)\in E\}$. Recall (\cite{MR2007488})
that a \emph{coarse structure }on a set $X$ is a collection $\mathcal{E}=\{E_{i}\}_{i\in I}$,
whose elements are called \emph{controlled sets}, where each $E_{i}$
is a relation $E_{i}\subseteq X\times X$, such that the following
axioms are satisfied.
\begin{itemize}
\item Reflexivity: The diagonal $\Delta=\{(x,x)\mid x\in X\}$ is in $\mathcal{E}$.
\item Downward Saturation: If $E_{i}\in\mathcal{E}$ and $F\subseteq E_{i}$
then $F\in\mathcal{E}$.
\item Upward Saturation: $\mathcal{E}$ is closed under taking finite unions. 
\item Composition Stability: if $E_{1},E_{2}\in\mathcal{E}$ then $E_{1}\circ E_{2}\in\mathcal{E}$. 
\item Symmetry: if $E\in\mathcal{E}$ then $E^{-1}\in\mathcal{E}$. 
\end{itemize}
A \emph{coarse space }is then a set $X$ together with a course structure
$\mathcal{E}$ on it. The archetypical example of a coarse space is
the \emph{bounded coarse structure} associated to an ordinary metric
space $(X,d)$ where the controlled sets are all subsets $E\subseteq X\times X$
such that $\sup\{d(x,y)\mid(x,y)\in E\}$ is finite. To adapt this
definition to the setting of metric $1$-spaces we first reformulate
the definition of coarse space to obtain an equivalent axiomatization
that does not mention symmetry. The steps we take are analogous to
those taken above on the way to the symmetry-free axiomatization of
metric space. 

Given a relation $E\subseteq X\times X$, let $E^{\star}$ be the
union of the sets

\[
\{(y,z)\in X\times X\mid\exists x\in X\quad(x,y),(x,z)\in E)\}
\]
and 
\[
\{(x,y)\in X\times X\mid\exists z\in X\quad(x,z),(y,z)\in E\}.
\]
We now obtain the following coarse version of Lemma \ref{lem:rest plus ref is full}.
\begin{lem}
If $\mathcal{E}$ is a collection of relations on a set $X$ that
satisfies reflexivity, upward and downward saturation, and composition
stability then $\mathcal{E}$ is a coarse structure if, and only if,
$\mathcal{E}$ is $\star$ closed (i.e., if $E\in\mathcal{E}$ then
$E^{\star}\in\mathcal{E}$).\end{lem}
\begin{proof}
We need to show that, under the given assumptions, symmetry is satisfied
if, and only if, $\mathcal{E}$ is $\star$ closed. If symmetry holds
then noting that for every $E\in\mathcal{E}$ holds that $E^{\star}\subseteq E\circ E^{-1}\cup E^{-1}\circ E$,
shows that $\mathcal{E}$ is $\star$ closed. Conversely, if $\mathcal{E}$
is $\star$ closed then, given $E\in\mathcal{E}$, form first $E_{0}=\Delta\cup E$
and then note that $E^{-1}\subseteq E_{0}^{\star}$, to finish the
proof. 
\end{proof}
Thus, a coarse space can equivalently be defined as a collection $\mathcal{E}$
of relations on a set $X$ which satisfies reflexivity, upward and
downward saturation, composition stability, and $\star$ stability.
It is this formulation that is the appropriate one to generalize to
metric $1$-spaces. Given $E,E_{1},E_{2}$, sets of arrows in a fixed
category, we write 
\[
E_{1}\circ E_{2}=\{\psi_{1}\circ\psi_{2}\mid\psi_{1}\in E_{1},\psi_{2}\in E_{2}\}
\]
 and 
\[
E^{\star}=\{\psi\in Arr(\mathscr{C})\mid\exists\varphi\in E\quad\psi\circ\varphi\in E\}\cup\{\psi\in Arr(\mathscr{C})\mid\exists\varphi\in E\quad\varphi\circ\psi\in E\}.
\]

\begin{defn}
A \emph{coarse structure }on a category $\mathscr{C}$ is a collection
$\mathcal{E}=\{E_{i}\}_{i\in I}$, where each $E_{i}$ is a set of
arrows in $\mathscr{C}$, called a \emph{controlled set}, such that
the following axioms hold.\end{defn}
\begin{itemize}
\item Reflexivity: The set $\Delta=\{id_{x}:x\to x\mid x\in ob(\mathscr{C})\}$
is in $\mathcal{E}$. 
\item Downward Saturation: If $E\in\mathcal{E}$ and $F\subseteq E$ then
$F\in\mathcal{E}$.
\item Upward Saturation: $\mathcal{E}$ is closed under taking finite unions. 
\item Composition Stability: If $E_{1},E_{2}\in\mathcal{E}$ then $E_{1}\circ E_{2}\in\mathcal{E}$. 
\item $\star$ Stability: If $E\in\mathcal{E}$ then $E^{\star}\in\mathcal{E}$.
\end{itemize}
A \emph{coarse $1$-space }is a category $X$ together with a coarse
structure on it. One can easily show that given a metric $1$-space
$X$, defining $\mathcal{E}$ to consist of all sets $E$ of arrows
in $X$ such that $\sup\{w(\psi)\mid\psi\in E\}$ is finite endows
the underlying category $X$ with a coarse structure which is called
the \emph{bounded coarse structure }of the metric $1$-space $X$. 

It is evident that coarse $1$-spaces can be seen as an extension
of coarse spaces via the construction of indiscrete categories analogously
to the case of metric $1$-spaces described above.

We close this section by proving a metrizability theorem for coarse
$1$-spaces. A coarse $1$-space is \emph{metrizable }if it is the
bounded coarse structure of some metric $1$-space. A \emph{generating
set }for a coarse $1$-space $(X,\mathcal{E})$ is a collection $\{E_{j}\}_{j\in J}\subseteq\mathcal{E}$
such that every controlled set $E\in\mathcal{E}$ is contained in
some $E_{i}$. 
\begin{thm}
A coarse $1$-space $(X,\mathcal{E})$ is metrizable if, and only
if, it admits a countable generating set.\end{thm}
\begin{proof}
If $X$ is metrizable then  defining, for each $n\ge0$, the set $E_{n}=\{\psi\in Arr(X)\mid w(\psi)\le n\}$
gives a countable generating set for the bounded coarse structure.
To prove the converse assume a countable generating set $\{E_{n}\}_{n=0}^{\infty}$
is given, and define the sets $F_{0}=\Delta$ and $F_{n+1}=F_{n}^{\star}\cup F_{n}\circ F_{n}\cup E_{n}\cup E_{n}^{\star}$,
for each $n\ge0$. Note that in general, if $\Delta\subseteq E$ then
$E\subseteq E^{\star}$ and thus it follows that $F_{n}\subseteq F_{n+1}$
for every $n\ge0$. Clearly, the set $\{F_{n}\}_{n=0}^{\infty}$ is
a generating set for the coarse structure $\mathcal{E}$. Now, for
each arrow $\psi\in Arr(X)$, define $w(\psi)=\inf\{n\in\mathbb{N}\mid\psi\in F_{n}\}$.
Note that $w(\psi)=0$ if, and only if, $\psi$ is an identity arrow.
To verify the full triangle inequality it suffices to only consider
non-identity arrows $\varphi$ and $\psi$. To that end assume that
$w(\varphi)=n\le m=w(\psi)$. Then $\varphi,\psi\in F_{m}$ and so,
by construction, $\varphi\circ\psi\in F_{m+1}$ proving that $w(\varphi\circ\psi)\le m+1\le m+n$.
This, and a similar argument, establish the restricted triangle inequality.
The full triangle inequality follows by considering cases such as:
$w(\varphi\circ\psi)=k\le m=w(\varphi)$, from which follows that
$\varphi\circ\psi,\varphi\in F_{m}$ and so $\psi\in F_{m+1}$, proving
that $w(\psi)\le m+1\le m+k$. We thus obtain the metric $1$-space
$(X,w)$ and it is trivial to verify that the bounded coarse structure
for this metric $1$-space coincides with $(X,\mathcal{E})$. 
\end{proof}

\section{\label{sec:convergence}convergence}

The notions of convergence introduced in this section are a fusion
of concepts from the theory of limits in ordinary metric spaces and
in categories. Since a category is a more elaborate structure than
a mere set there are different ways to proceed. We present here two
notions of convergence in metric $1$-spaces which will be used below
to prove some fundamental results on metric $1$-spaces. The comparison
with the notion of topological limit will be quite self evident. It
is also interesting to compare and contrast with the categorical notion
of a (co)limit and thus we first consider categorical (co)limits of
a special kind. Due to categorical duality all concepts of category
theory come in pairs, commonly prefixed by 'left' and 'right' with
the prefix 'co' used to signify an application of duality. As a result,
the notions of sequence and series we introduce have duals as well.
This is in accordance with current literature on quasi-metric spaces
(e.g., \cite{rutten1998weighted}) where convergence and all other
topological notions appear in two variants. These are usually called
forward and backward (e.g., forward/backward convergence, forward/backward
Cauchy etc.) and this is the terminology we adopt. 
\begin{rem*}
We point out that in \cite{rutten1998weighted} it is shown that limits
in Lawvere spaces are related to the notion of weighted limits in
enriched category theory. Our presentation below does not involve
enrichment. 
\end{rem*}

\subsection{Sequences and series}

Let $\mathbb{N}_{\bullet}$ be the category whose objects are the
natural numbers together with an object $\bullet$ such that, besides
the identity arrows, there is in the category, for every natural number
$n$, precisely one arrow from $\bullet$ to $n$. The category $\mathbb{N}_{\bullet}$
can be depicted diagrammatically as

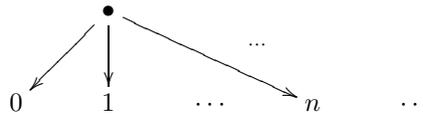
\[
\xymatrix{ & \bullet\ar[dl]\ar[d]\ar[drr]^{\quad\quad\cdots}\\
0 & 1 & \cdots & n & \cdots
}
\]
Dually, let $\mathbb{N}^{\bullet}=\mathbb{N}_{\bullet}^{op}$ be the
opposite category whose diagram is
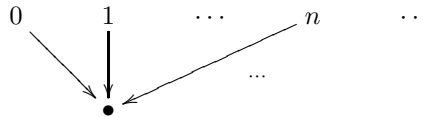
\[
\xymatrix{0\ar[dr] & 1\ar[d] & \cdots & n\ar[dll]^{\quad\quad\cdots} & \cdots\\
 & \bullet
}
\]

\begin{defn}
Let $\mathscr{C}$ be a category. A \emph{forward sequence }in the
category $\mathscr{C}$ is a functor $\mathbb{N}_{\bullet}\to\mathscr{C}$.
A \emph{backward sequence }in the category $\mathscr{C}$ is a functor
$\mathbb{N}^{\bullet}\to\mathscr{C}$. 
\end{defn}
Clearly a forward sequence $\mathbb{N}_{\bullet}\to\mathscr{C}$ amounts
to specifying a family $\{\psi_{n}:c\to c_{n}\}_{n=0}^{\infty}$ of
arrows in $\mathscr{C}$ with a common domain. If we wish to make
the domain explicit we will speak of a forward sequence \emph{from
an object $c$}. Similarly, a backward sequence $\mathbb{N}^{\bullet}\to\mathscr{C}$
consists of a family $\{\psi_{n}:c_{n}\to c\}_{n=0}^{\infty}$ of
arrows in $\mathscr{C}$ with a common codomain which can be made
explicit by speaking of a backward sequence \emph{to an object $c$. }

Let $\mathbb{N}_{\rightarrow}$ be the category whose objects are
the natural numbers and, besides the identity arrows, there is an
arrow from $n$ to $m$ if, and only if, $n<m$. Dually, let $\mathbb{N}_{\leftarrow}=\mathbb{N}_{\rightarrow}^{op}$
be the opposite category. The respective diagrams of these categories
are
\[
\xymatrix{0\ar[r] & 1\ar[r] & \cdots\ar[r] & n\ar[r] & \cdots}
\]
and
\[
\xymatrix{\cdots\ar[r] & n\ar[r] & \cdots\ar[r] & 1\ar[r] & 0}
\]
(the identities and compositions are omitted from the diagrams). 
\begin{defn}
Let $\mathscr{C}$ be a category. A \emph{forward series }in the category
$\mathscr{C}$ is a functor $\mathbb{N}_{\rightarrow}\longrightarrow\mathscr{C}$.
A \emph{backward series }in the category $\mathscr{C}$ is a functor
$\mathbb{N}_{\leftarrow}\longrightarrow\mathscr{C}$. 
\end{defn}
Clearly, a forward series amounts to a sequence of arrows $\{\psi_{n}\}_{n=0}^{\infty}$
such that for each $n\ge0$ the domain of $\psi_{n+1}$ is equal to
the codomain of $\psi_{n}$. In other words, a sequence of arrows
$\{\psi_{n}\}_{n=0}^{\infty}$ is a forward series if, and only if,
for each $n\ge0$ the composition $\psi_{n}\circ\psi_{n-1}\circ\cdots\circ\psi_{0}$
exists. Similarly, a sequence $\{\psi_{n}\}_{n=0}^{\infty}$ of arrows
is a backward series precisely when for each $n\ge0$ the composition
$\psi_{0}\circ\cdots\circ\psi_{n-1}\circ\psi_{n}$ exists. 

Our notation is meant to resonate with the familiar concepts of sequences
and series in, e.g., $\mathbb{R}$. However, there are marked differences
which show up below. For instance, the two notions are generally not
interchangeable. Series can be related to sequences by the evident
construction of partial compositions as follows. Let $\{\psi_{n}\}_{n=0}^{\infty}$
be a forward series in $\mathscr{C}$. Define for each $n\ge0$ the
arrow $\varphi_{n}=\psi_{n}\circ\cdots\circ\psi_{0}$. The resulting
sequence $\{\varphi_{n}\}_{n=0}^{\infty}$ is called the forward sequence
of \emph{partial compositions }associated to the forward series $\{\psi_{n}\}_{n=0}^{\infty}$.
Similarly, one can associate to a backward series a backward sequence.
However, it will be evident below that convergence of the forward
(respectively backward) series is generally stronger than convergence
of the associated forward (respectively backward) sequence of partial
compositions. It is also evident that not every forward (respectively
backward) sequence can so be obtained from a forward (respectively
backward) series.

\subsection{Pushouts, Pullbacks, and transfinite compositions}

Before presenting the definitions of limits for the concepts introduced
above we define, for the sake of completeness, categorical limits
and colimits of sequences and series first. 
\begin{defn}
Let $\mathscr{C}$ be a category and consider a forward sequence $\mathbb{N}_{\bullet}\to\mathscr{C}$
represented by the solid arrows from the object $c$ in the diagram
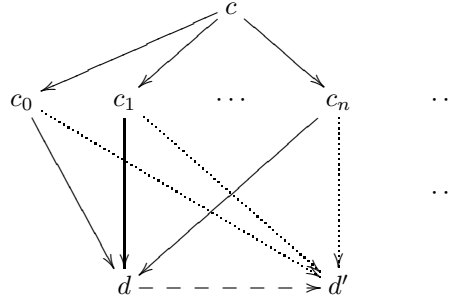
\[
\xymatrix{ &  & c\ar[dll]\ar[dl]\ar[dr]\\
c_{0}\ar[ddr]\ar@{..>}[ddrrr] & c_{1}\ar[dd]\ar@{..>}[ddrr] & \cdots & c_{n}\ar[ddll]\ar@{..>}[dd] & \cdots\\
 &  &  &  & \cdots\\
 & d\ar@{-->}[rr] &  & d'
}
\]
A \emph{cone }over the forward sequence is a sequence of arrows to
an object $d$ such that the diagram of solid arrows commutes. Such
a cone is called a \emph{weak pushout }of the forward sequence if
given any other cone (over the same forward sequence), represented
by the dashed arrows to $d'$, there exists a \emph{mediating }arrow
from $d$ to $d'$ that makes the entire diagram commute. A \emph{pushout
}is then a weak pushout such that each such mediating arrow is unique. 
\end{defn}
It can easily be shown that a pushout of a forward sequence, if it
exists, is unique up to an isomorphism.

The dual notion is that of a pullback of a backward sequence, obtained
by formally reversing all the arrows in the above definition.
\begin{defn}
\label{Def:categorical tran comp}Let $\mathscr{C}$ be a category
and $\{\psi_{n}:x_{n}\to x_{n+1}\}_{n=0}^{\infty}$ a series of arrows
in it. A \emph{categorical weak transfinite composition }of the series
is an object $x_{\infty}\in ob(\mathscr{C})$ together with arrows
$\mu_{n}:x_{n}\to x_{\infty}$ forming the solid commutative diagram
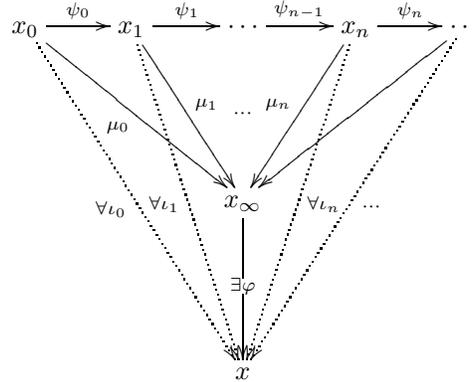
\[
\xymatrix{x_{0}\ar[r]^{\psi_{0}}\ar[ddrr]_{\mu_{0}}\ar@{..>}[ddddrr]_{\forall\iota_{0}} & x_{1}\ar[r]^{\psi_{1}}\ar[ddr]^{\mu_{1}}\ar@{..>}[ddddr]_{\forall\iota_{1}} & \cdots\ar[r]^{\psi_{n-1}}\ar@{}[dd]|{\cdots} & x_{n}\ar[ddl]_{\mu_{n}}\ar[r]^{\psi_{n}}\ar@{..>}[ddddl]^{\forall\iota_{n}} & \cdots\ar@{..>}[ddddll]^{\cdots}\ar[ddll]\\
\\
 &  & x_{\infty}\ar[dd]|{\exists\varphi}\\
\\
 &  & x
}
\]
with the property that if $x\in ob(\mathscr{C})$ is any object and
the dotted arrows $\iota_{n}:x_{n}\to x$ are any arrows that form
a commutative diagram with the given series $\{\psi_{n}\}_{n=0}^{\infty}$
then there exists a mediating arrow $\varphi:x_{\infty}\to x$ such
that $\iota_{k}=\varphi\circ\mu_{k}$ for all $k\ge0$. A \emph{categorical
transfinite composition }is then a weak categorical transfinite composition
for which each such mediating arrow is unique. 
\end{defn}
Again, the dual notion of (weak) categorical transfinite composition
of a backward series is obtained by formally reversing all arrows
and again we omit the details. We remark that the notions of transfinite
composition, pushout, and pullback are special cases of general categorical
colimits and limits (see, e.g., \cite{MR1712872}).

\subsection{Limits of sequences }

Note that a (weak) pushout of a forward sequence $\{\psi_{n}:c\to c_{n}\}_{n=0}^{\infty}$
is highly sensitive to finite changes in the sequence. Indeed, changing
just one of the arrows in the sequence can change a sequence that
has a pushout to one that does not. Our definition of limits below
removes this sensitivity by adapting the definition of cone so that
any limit will essentially depend only on what happens 'towards the
end of the sequence'. 
\begin{defn}
\label{Def Forward Limiting arrow of sequence}Let $S:\mathbb{N}_{\bullet}\to X$
be a forward sequence in a metric $1$-space $X$, represented in
the diagram 
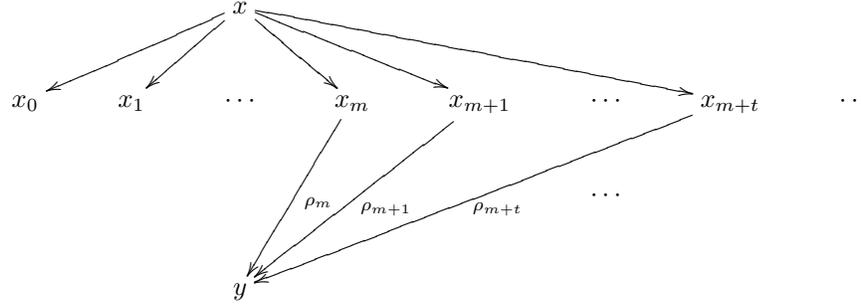
\[
\xymatrix{ &  & x\ar[dll]\ar[dl]\ar[dr]\ar[drr]\ar[drrrr]\\
x_{0} & x_{1} & \cdots & x_{m}\ar[ddl]^{\rho_{m}} & x_{m+1}\ar[ddll]^{\rho_{m+1}} & \cdots & x_{m+t}\ar[ddllll]^{\rho_{m+t}} & \cdots\\
 &  &  &  &  & \cdots\\
 &  & y
}
\]
by the arrows from the object $x$. An \emph{essential} \emph{cone
}over the forward sequence is given by arrows $\{\rho_{k}:x_{k}\to y\}_{k=m}^{\infty}$
to an object $y$ such that the diagram commutes. Such an essential
cone is called a \emph{forward limiting cone }of the forward sequence
if $\lim_{k\to\infty}w(\rho_{k})=0$. In that case, the arrow $x\to y$
(the common value of all compositions in the diagram) is called the
\emph{forward limiting arrow} associated to the forward limit and
is denoted, ambiguously, by $\lim_{n\to\infty}S$. 
\end{defn}
Applying categorical duality we obtain the definition of a \emph{backward
limiting arrow }$\lim_{n\to\infty}S$ of a backward sequence $S:\mathbb{N}^{\bullet}\to X$.
Once more, we omit the details. 

Addressing uniqueness requires the following notion. 
\begin{defn}
Consider two essential cones over the same forward sequence as in
the diagram (drawn from some index $m$ where both cones are defined)
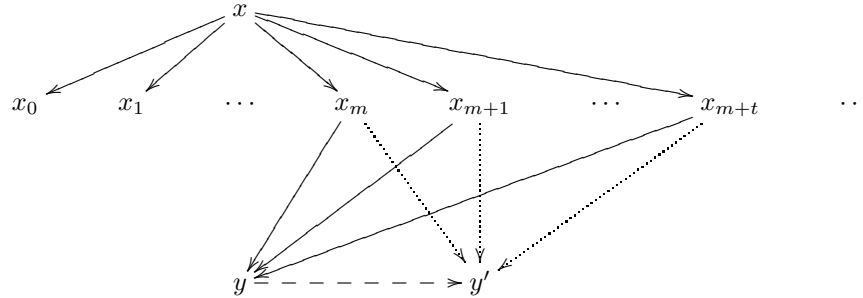
\[
\xymatrix{ &  & x\ar[dll]\ar[dl]\ar[dr]\ar[drr]\ar[drrrr]\\
x_{0} & x_{1} & \cdots & x_{m}\ar[ddl]\ar@{..>}[ddr] & x_{m+1}\ar[ddll]\ar@{..>}[dd] & \cdots & x_{m+t}\ar[ddllll]\ar@{..>}[ddll] & \cdots\\
\\
 &  & y\ar@{-->}[rr] &  & y'
}
\]
A \emph{factorization }of the cone to $y'$ through the cone to $y$
is a dashed arrow, called a \emph{mediating arrow,} such that for
infinitely many values $l\in\mathbb{N}$ holds that the triangle
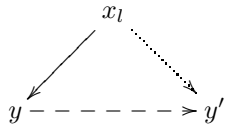
\[
\xymatrix{ & x_{l}\ar@{..>}[dr]\ar[dl]\\
y\ar@{-->}[rr] &  & y'
}
\]
commutes. If a factorization between two cones exists then we say
that they are \emph{compatible}. \end{defn}
\begin{lem}
\label{lem: uniqueness of sequence limit}Let $X$ be a metric $1$-space,
$S:\mathbb{N}_{\bullet}\to X$ a forward sequence from $x$, and $C_{\mu}$
and $C_{\nu}$ two forward limiting cones. If $y\to y'$ is a mediating
arrow from $C_{\mu}$ to $C_{\nu}$ then $w(y\to y')=0$. Dually,
mediating arrows between backward limiting cones of a backward sequence
have weight $0$.\end{lem}
\begin{proof}
The proof is a straightforward application of the full triangle inequality. \end{proof}
\begin{cor}
If $X$ is non-degenerate then a forward limiting arrow, if it exists
is unique within compatible cones. To be more precise, if $X$ is
non-degenerate then two forward limiting arrows (of the same forward
sequence) with compatible cones are equal. Dually, backward limiting
arrows in non-degenerate metric $1$-spaces are similarly essentially
unique. \end{cor}
\begin{example}
In any metric $1$-space a constant forward (respectively backward)
sequence $\{\psi_{n}\}_{n=0}^{\infty}$, $\psi_{n}=\psi$, has $\psi$
as a forward (respectively backward) limiting arrow. If $S$ is a
set then we saw above that ordinary metric structures on $S$ correspond
to metric structures on the indiscrete category $I_{S}$. Forward
and backward sequences in $I_{S}$ are then essentially the same as
sequences in $S$ and convergence in $S$ and in $I_{S}$ agree. More
explicitly, if $\{\psi_{n}:y\to z_{n}\}_{n=0}^{\infty}$ is a forward
sequence in $I_{S}$ then its limit exists in $I_{S}$ if, and only
if, the sequence of points $\{z_{n}\}_{n=0}^{\infty}$ in $S$ converges
in the ordinary sense. In that case, a limiting object of the sequence
is the unique arrow from $y$ to the limit point of $\{z_{n}\}_{n=0}^{\infty}$
in $S$. A similar remark holds for backward limits. Moreover, starting
with any convergent sequence of points in $S$ one may obtain its
limit as a limiting object in $I_{S}$ of both a forward and a backward
sequence in $I_{S}$.\end{example}
\begin{rem}
All of the results that follow, when specialized to metric $1$-spaces
on an indiscrete category, relate to familiar notions in the ordinary
theory of limits. One of the aims of this work is to show that the
standard theory extends, in just this sense, to our more general setting.
For the sake of keeping the presentation short we will not point out
precisely how each result below extends familiar results. More often
than not, it is self evident. \end{rem}
\begin{prop}
\label{prop:Internal continuity sequences}Let $X$ be a metric $1$-space.
If the forward sequence $S:\mathbb{N}_{\bullet}\to X$ admits a limit
then 
\[
\lim_{n\to\infty}w(S(\bullet\to n))=w(\lim_{n\to\infty}S).
\]
Dually, if the backward sequence $S:\mathbb{N}^{\bullet}\to X$ admits
a limit then
\[
\lim_{n\to\infty}w(S(n\to\bullet))=w(\lim_{n\to\infty}S).
\]
\end{prop}
\begin{proof}
Since the assertions are dual it suffices to prove the first one.
Let $\mu:x\to y$ be a forward limiting arrow with forward limiting
cone given by arrows $\rho_{m}$ as in the diagram in Definition \ref{Def Forward Limiting arrow of sequence}.
Since for almost all $k$ the diagram
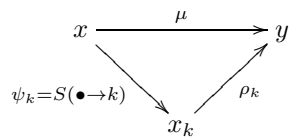
\[
\xymatrix{x\ar[dr]_{\psi_{k}=S(\bullet\to k)}\ar[rr]^{\mu} &  & y\\
 & x_{k}\ar[ur]_{\rho_{k}}
}
\]
commutes, the full triangle inequality implies that 
\[
w(\psi_{k})-w(\rho_{k})\le w(\mu)\le w(\psi_{k})+w(\rho_{k})
\]
and thus
\[
w(\mu)-w(\rho_{k})\le w(\psi_{k})\le w(\mu)+w(\rho_{k}).
\]
The result now follows since $\lim_{k\to\infty}w(\rho_{k})=0$. \end{proof}
\begin{rem}
For ordinary metric spaces the uniqueness of the limit is easily proven
but depends in an essential way on the axiom of symmetry. Indeed,
in a Lawvere space one may consider two kinds of convergence (see
\cite{rutten1998weighted}) none of which exhibits uniqueness of the
limiting point, not even after identifying points with distance $0$.
As shown above, in our approach the essential uniqueness of the limiting
arrow of a sequence is a result of an interplay between the underlying
categorical essential cones (i.e., their compatibility) and the metric
structure defined on the category (with the full triangle inequality
playing an essential role). Viewed this way, the uniqueness of the
limit in ordinary metric spaces is a consequence of the fact that
any two forward (respectively backward) limiting cones over the same
forward (respectively backward) sequence in an indiscrete category
$I_{S}$ are compatible. 
\end{rem}

\subsection{Limits of series}

The ability to compose arrows in a category naturally gives rise to
another limiting notion in metric $1$-spaces. 
\begin{defn}
Let $X$ be a metric $1$-space and $S:\mathbb{N}_{\rightarrow}\to X$
a forward series of arrows in it, depicted by the horizontal arrows
in the following diagram. 
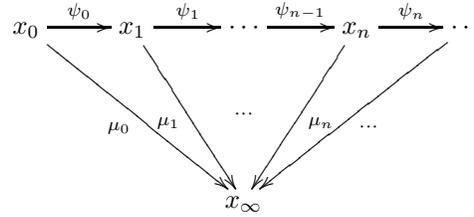
\[
\xymatrix{x_{0}\ar[r]^{\psi_{0}}\ar[ddrr]_{\mu_{0}} & x_{1}\ar[r]^{\psi_{1}}\ar[ddr]_{\mu_{1}} & \cdots\ar[r]^{\psi_{n-1}}\ar@{}[dd]|{\cdots} & x_{n}\ar[ddl]^{\mu_{n}}\ar[r]^{\psi_{n}} & \cdots\ar[ddll]^{\cdots}\\
\\
 &  & x_{\infty}
}
\]
A \emph{forward limit }of the series is an object $x_{\infty}\in ob(X)$
together with arrows $\mu_{n}:x_{n}\to x_{\infty}$, forming a \emph{forward
limiting cone }in the sense that the diagram commutes and $\lim_{n\to\infty}w(\mu_{n})=0$.
We then say that the series \emph{converges} to $\mu_{0}:x_{0}\to x_{\infty}$
which we denote by $\bigcirc_{n=0}^{\infty}\psi_{n}=\mu_{0}$. 
\end{defn}
As usual, the dual notion of a limit of a backward series is obtained
by reversing arrows.

Note that if a forward (respectively backward) series $S:\mathbb{N}_{\rightarrow}\to X$
converges to $\mu_{0}$ then its associated forward (respectively
backward) sequence of partial compositions converges to $\mu_{0}$
(the reverse implication is not generally true). The following is
thus an immediate consequence of Lemma \ref{lem: uniqueness of sequence limit}.
\begin{lem}
\label{lem:mediating is 0 for series}Given a forward (respectively
backward) series that converges to $\mu$ and $\nu$, as witnessed
by cones $C_{\mu}$ and $C_{\nu}$, if $\varphi$ is a mediating arrow
from $C_{\mu}$ to $C_{\nu}$ (as in Definition \ref{Def:categorical tran comp})
then $w(\varphi)=0$.\end{lem}
\begin{rem}
Note that the uniqueness within compatible cones of the limit of a
forward (respectively backward) series follows from a less stringent
condition than the existence of a mediating arrow as in the preceding
lemma. The precise relation between series and sequences can be elaborated
much more but these subtleties play no significant role in this work
and are thus neglected. 
\end{rem}
The following simple result, whose proof is omitted, shows again the
importance of the full triangle inequality to obtain results that
echo the fundamentals of the theory of convergence in ordinary metric
spaces. 
\begin{prop}
If a forward (respectively backward) series $\{\psi_{n}\}_{n=0}^{\infty}$
converges then $\lim_{n\to\infty}w(\psi_{n})=0$. Moreover, if the
series forward converges then 
\[
\lim_{n\to\infty}w(\psi_{n}\circ\psi_{n-1}\circ\cdots\circ\psi_{0})=w(\bigcirc_{n=0}^{\infty}\psi_{n})
\]
while if the series backward converges then 
\[
\lim_{n\to\infty}w(\psi_{0}\circ\psi_{1}\circ\cdots\circ\psi_{n})=w(\bigcirc_{n=0}^{\infty}\psi_{n}).
\]

\end{prop}
The following lemma will be used in the proof of the Banach fixed
point theorem below. Again a proof is left for the reader. 
\begin{lem}
\label{lem:truncations of series}Let $S:\mathbb{N}_{\rightarrow}\to X$
be a forward series in a metric $1$-space $X$ with forward limiting
cone $\{\psi_{n}\}_{n=0}^{\infty}$. If $k\ge0$ is any natural number
then the forward series $S^{k}:\mathbb{N}_{\rightarrow}\to X$, given
by $S^{k}(n)=S(n+k)$ and extended uniquely to arrows, has $\{\psi_{m}\}_{m=k}^{\infty}$
as forward limiting cone. 
\end{lem}
The dual notion relates truncations of a backward limiting cone of
a backward series to truncations of the backward series.

\section{\label{sec:Continuity}Continuity}

We present two notions of continuity at an arrow and some related
concepts. As we show below, the two concepts of continuity coincide
when the domain is suitably compact. We mention as well that there
is a third natural notion of continuity at an arrow which is obtained
by considering decompositions of $\psi$ as $\rho\varphi\tau$. The
theory of this kind of limit is similar to the other two and there
are interrelations between all notions. However, to keep the presentation
more concise we do not consider this third possibility here. 

In what follows it will be useful to introduce the following convention.
For an arrow $x\to y$ in a metric $1$-space and a non-negative real
number $t$ we write $\xymatrix{x\ar[r]^{t} & y}
$ to indicate that $w(x\to y)<t$.
\begin{defn}
$F$ is said to be \emph{forward continuous} at the arrow $\psi:x\to z$
if for every $\epsilon>0$ there is a $\delta>0$ such that $F$ sends
any commuting diagram of the form
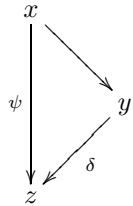
\[
\xymatrix{x\ar[dd]_{\psi}\ar[rd]\\
 & y\ar[dl]^{\delta}\\
z
}
\]
to a diagram of the form
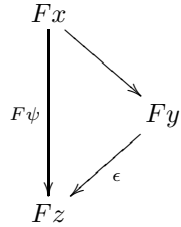
\[
\xymatrix{Fx\ar[dd]_{F\psi}\ar[rd]\\
 & Fy\ar[dl]^{\epsilon}\\
Fz
}
\]
We say that $F$ is \emph{forward continuous} if it is forward continuous
at every arrow in $X$. We say that $F$ is\emph{ uniformly continuous}
if for every $\epsilon>0$ there is a $\delta>0$ such that $w(F(\psi))<\epsilon$
holds for every $\psi$ with $w(\psi)<\delta$. 
\end{defn}
The notions of a functor being \emph{backward continuous }at an arrow
and \emph{backward continuous} are defined by duality. Notice that
the concept of uniform continuity is self-dual. 
\begin{defn}
Let $F:X\to Y$ be a functor between metric $1$-spaces. We say that\end{defn}
\begin{itemize}
\item $F$ is\emph{ forward continuous at an object $x_{0}\in ob(X)$ }if
for every $\epsilon>0$ there exists $\delta>0$ such that $w(\psi)<\delta$
implies $w(F(\psi))<\epsilon$ for every $\psi:x\to x_{0}$.
\item $F$ is \emph{backward continuous }at an object \emph{$x_{0}\in ob(X)$
}if for every $\epsilon>0$ there exists $\delta>0$ such that $w(\psi)<\delta$
implies $w(F(\psi))<\epsilon$ for every $\psi:x_{0}\to x$.
\item $F$ is \emph{object forward (respectively backward) continuous }if
it is forward (respectively backward) continuous at every object $x_{0}\in ob(X)$.\end{itemize}
\begin{prop}
Let $F:X\to Y$ be a functor between metric $1$-spaces and $\psi:x\to z$
an arrow in $X$. \end{prop}
\begin{itemize}
\item Forward continuity at $z$ implies forward continuity at $\psi:x\to z$. 
\item Backward continuity at $x$ implies backward continuity at $\psi:x\to z$. \end{itemize}
\begin{proof}
Straightforward. \end{proof}
\begin{prop}
\label{prop:unif. cont. implies obj. cont.}Let $F:X\to Y$ be a functor
between metric $1$-spaces. If $F$ is uniformly continuous then $F$
is object forward and object backward continuous.\end{prop}
\begin{proof}
To prove the forward case (the backward case being obtained by duality)
fix an object $z\in ob(X)$. Given $\epsilon>0$ there is a $\delta>0$
such that any commuting square of the form 
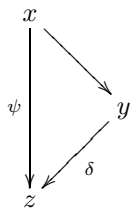
\[
\xymatrix{x\ar[dd]_{\psi}\ar[rd]\\
 & y\ar[dl]^{\delta}\\
z
}
\]
yields, upon applying $F$, a square of the form 
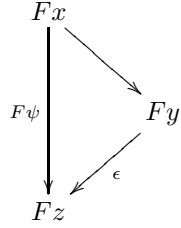
\[
\xymatrix{Fx\ar[dd]_{F\psi}\ar[rd]\\
 & Fy\ar[dl]^{\epsilon}\\
Fz
}
\]
In particular thus, given an arrow $\xymatrix{y\ar[r]^{\varphi} & z}
$, with $w(\varphi)<\delta$, consider the commuting diagram
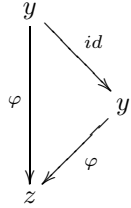
\[
\xymatrix{y\ar[dd]_{\varphi}\ar[rd]^{id}\\
 & y\ar[dl]^{\varphi}\\
z
}
\]
and apply $F$ to it to obtain the diagram
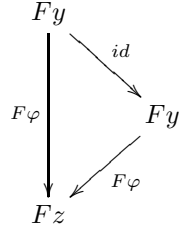
\[
\xymatrix{Fy\ar[dd]_{F\varphi}\ar[rd]^{id}\\
 & Fy\ar[dl]^{F\varphi}\\
Fz
}
\]
proving that $w(F(\varphi))<\epsilon$ and so $F$ is forward continuous
at $z$. 
\end{proof}
Thus we see that forward object continuity lies between forward continuity
and uniform continuity. A similar remark holds for backward continuity. 

The next result is the metric analogue of the fact that a functor
$F:\mathscr{C}\to\mathscr{D}$ must map an isomorphism to an isomorphism.
\begin{lem}
\label{lem:cont preserves 0 weight}Let $X$ and $Y$ be metric $1$-spaces.
If $F:X\to Y$ is either forward or backward continuous at the arrow
$\psi$ and $w(\psi)=0$ then $w(F(\psi))=0$.\end{lem}
\begin{proof}
Let us assume $F$ is forward continuous (the proof in the backward
case is similar) and let $\epsilon>0$ with a corresponding $\delta>0$
obtained from forward continuity at $\psi$. Since we can always write
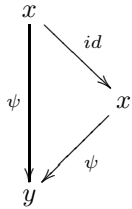
\[
\xymatrix{x\ar[dd]_{\psi}\ar[dr]^{id}\\
 & x\ar[dl]^{\psi}\\
y
}
\]
and since $w(\psi)<\delta$ we may conclude that $w(F\psi)<\epsilon$,
and the result follows. 
\end{proof}
Lastly, the following result is the expected claim that forward (respectively
backward) continuous functor preserve forward (respectively backward)
limits of both sequences and series. The proof is trivial and thus
omitted. 
\begin{thm}
\label{thm:mapping preserves limits}For a functor $F:X\to Y$ between
metric $1$-spaces the following hold.\end{thm}
\begin{itemize}
\item For every forward sequence $S:\mathbb{N}_{\bullet}\to X$, a forward
limiting arrow $\mu$, and a forward limiting cone $C$ holds that
if $F$ is forward continuous at $\mu$ then $F(\mu)$ is a forward
limiting arrow of the forward sequence $\mathbb{N}_{\bullet}\to X\to Y$
with forward limiting cone $F(C)$. 
\item For every forward series $S:\mathbb{N}_{\rightarrow}\to X$, a forward
limiting arrow $\mu$, and a forward limiting cone $C$ holds that
if $F$ is forward continuous at $\mu$ then $F(\mu)$ is a forward
limiting arrow of the forward series $\mathbb{N}_{\rightarrow}\to X\to Y$
with forward limiting cone $F(C)$.
\item For every backward sequence $S:\mathbb{N}^{\bullet}\to X$, a backward
limiting arrow $\mu$, and a backward limiting cone $C$ holds that
if $F$ is backward continuous at $\mu$ then $F(\mu)$ is a backward
limiting arrow of the backward sequence $\mathbb{N}^{\bullet}\to X\to Y$
with backward limiting cone $F(C)$. 
\item For every backward series $S:\mathbb{N}_{\leftarrow}\to X$, a backward
limiting arrow $\mu$, and a backward limiting cone $C$ holds that
if $F$ is backward continuous at $\mu$ then $F(\mu)$ is a backward
limiting arrow of the backward series $\mathbb{N}_{\leftarrow}\to X\to Y$
with backward limiting cone $F(C)$.
\end{itemize}

\section{\label{sec:Fundamental-results}Fundamental results}

The aim of this section is to show that three fundamental results
from the theory of ordinary metric spaces extend to metric $1$-spaces.
The results we generalize are that continuity implies uniform continuity
when the domain is compact, that for compact spaces mapping spaces
are metric spaces, and the Banach fixed point theorem for contractive
mappings on complete spaces. We note immediately that due to the different
notions of convergence (i.e., sequences vs. series) and the forward/backward
duality there are several ways to interpret how compactness and completeness
should be generalized. The definitions we present here are chosen
to best fit the proofs of the three theorems we aim at. Other possibilities
exist and are useful in different situations. We also mention that
the proofs use most of the results recounted above and thus all depend
essentially on the full triangle inequality. In particular, the proofs
below make no use of symmetry and thus, upon specializing them to
metric $1$-spaces with indiscrete underlying categories, we obtain
proofs of the classical versions of the mentioned results that do
not make explicit use of symmetry. After presently stating the relevant
definitions of compactness and completeness we proceed to prove the
three results in the just stated order. 
\begin{defn}
A metric $1$-space $X$ is said to be \emph{forward compact }(respectively
\emph{backward compact) if} every forward (respectively backward)
sequence has a convergent subsequence.
\end{defn}
Here the meaning of 'subsequence' is the obvious one.
\begin{defn}
\label{Def:obj compac}A metric $1$-space $X$ is said to be \emph{object
compact }if given any sequence  $\{x_{n}\}_{n=1}^{\infty}$ of objects
in $ob(X)$ there exists an object $x_{0}\in ob(X)$ such that for
every $\epsilon>0$ and $N\in\mathbb{N}$ there exist $m,n>N$ and
arrows $\psi:x_{n}\to x_{0}$ and $\varphi:x_{0}\to x_{m}$ such that
both $w(\psi)<\epsilon$ and $w(\varphi)<\epsilon$.
\end{defn}
Of course this definition could be refined to include backward and
forward versions. For the proofs below such generality is not needed.
\begin{defn}
Let $X$ be a metric $1$-space. A forward series $\{\psi_{n}:x_{n}\to x_{n+1}\}_{n=1}^{\infty}$
in $X$ is a\emph{ Cauchy forward series }if for every $\epsilon>0$
there exists $N\in\mathbb{N}$ such that for every $n>m>N$ holds
that $w(\psi_{n}\circ\psi_{n-1}\circ\cdots\circ\psi_{m})<\epsilon$.
Dually, a backward series $\{\psi_{n}:x_{n+1}\to x_{n}\}_{n=0}^{\infty}$
is a \emph{Cauchy backward series }if for every $\epsilon>0$ there
exists $N\in\mathbb{N}$ such that for every $n>m>N$ holds that $w(\psi_{m}\circ\psi_{m+1}\circ\cdots\circ\psi_{n})<\epsilon$. 
\end{defn}
We omit the details of the straightforward fact that any convergent
forward or backward series is Cauchy.
\begin{defn}
A metric $1$-space $X$ is said to be \emph{forward complete }(respectively
\emph{backward complete) }if every forward (respectively backward)
Cauchy series $\{\psi_{n}\}_{n=1}^{\infty}$ converges.  
\end{defn}

\subsection{Uniform continuity}
\begin{thm}
Let $F:X\to Y$ be a forward (respectively backward) continuous functor
between metric $1$-spaces. If $X$ is forward (respectively backward)
compact then $F$ is object forward (respectively backward) continuous.
If moreover $X$ is object compact then $F$ is uniformly continuous.\end{thm}
\begin{proof}
Assume that $F$ is forward continuous at every arrow but not forward
continuous at some  object $x\in ob(X)$. Thus, there exists an $\epsilon_{0}>0$
and, for every $n\in\mathbb{N}$, an arrow $\psi_{n}:x\to x_{n}$
such that $w(\psi_{n})<\frac{1}{n}$ and $w(F\psi_{n})\ge\epsilon_{0}$.
The arrows $\psi_{n}$ form a forward sequence which, by forward compactness
of $X$, we may assume converges to some arrow $\mu$. Proposition
\ref{prop:Internal continuity sequences} implies that $w(\mu)=0$
and together with Theorem \ref{thm:mapping preserves limits} we may
also conclude that $w(F\mu)\ge\epsilon_{0}$. We now arrive at a contradiction
with Lemma \ref{lem:cont preserves 0 weight}. It follows that $F$
is object forward continuous. 

Now, under the extra assumption that $X$ is object compact, suppose
that $F$ is not uniformly forward continuous. Then there exists an
$\epsilon_{0}>0$ and for every $n\in\mathbb{N}$ an arrow $\varphi_{n}:x_{n}\to z_{n}$
such that $w(\varphi_{n})<\frac{1}{n}$ and $w(F(\varphi_{n}))\ge\epsilon_{0}$.
By object compactness there is an object $x$ and arrows $\psi_{n}:x\to x_{n}$
with $w(\psi_{n})<\frac{1}{n}$. Since $F$ is forward continuous
at $x$ we may assume without loss of generality that $w(F(\psi_{n}))<\frac{\epsilon_{0}}{2}$
holds for all $n\ge1$. The full triangle inequality now yields that
$w(F(\varphi_{n}\circ\psi_{n}))>\frac{\epsilon_{0}}{2}$ while $w(\varphi_{n}\circ\psi_{n})<\frac{2}{n}$.
Using forward compactness we again obtain a contradiction with Lemma
\ref{lem:cont preserves 0 weight}, thus completing the proof. 
\end{proof}
Since uniform continuity is a self dual property we obtain:
\begin{cor}
\label{cor:compact domain implies for is bac}If $X$ is object compact
as well as forward and backward compact then a functor $F:X\to Y$
is forward continuous if, and only if, it is backward continuous. 
\end{cor}

\subsection{Mapping spaces}

To construct the mapping metric $1$-space of metric $1$-spaces we
introduce the notion of natural transformations between functors. 
\begin{defn}
Given categories $\mathscr{C}$ and $\mathscr{D}$ and functors $F,G:\mathscr{C}\to\mathscr{D}$
a, \emph{natural transformation }$\alpha:F\to G$ is a family $\{\alpha_{c}:F(c)\to G(c)\}_{c\in ob(\mathscr{C})}$
of arrows in $\mathscr{D}$ such that for every arrow $\psi:c\to c'$
in $\mathscr{C}$ the diagram
\[
\xymatrix{F(c)\ar[d]_{F(\psi)}\ar[r]^{\alpha_{c}} & G(c)\ar[d]^{G(\psi)}\\
F(c')\ar[r]^{\alpha_{c'}} & G(c')
}
\]
commutes. 
\end{defn}
It is well known that given categories $\mathscr{C}$ and $\mathscr{D}$
the collection of all functors $F:\mathscr{C}\to\mathscr{D}$ as objects
and natural transformations $\alpha:F\to G$ as arrows forms a category
known as the \emph{functor category }$\underline{\mathbf{Cat}}(\mathscr{C},\mathscr{D})$.
We refer to \cite{MR1712872} for more details and just recall here
that the composition (also known as vertical composition) of natural
transformations that turns $\underline{\mathbf{Cat}}(\mathscr{C},\mathscr{D})$
into a category is given, for natural transformations $\alpha:F\to G$
and $\beta:G\to H$, by the family of arrows 
\[
\{\xymatrix{F(c)\ar[r]^{\alpha_{c}} & G(c)\ar[r]^{\beta_{c}} & H(c)}
\}_{c\in ob(\mathscr{C})}.
\]

In the presence of a metric structure on the categories $\mathscr{C}$
and $\mathscr{D}$, any natural transformation $\alpha$ in $\underline{\mathbf{Cat}}(\mathscr{C},\mathscr{D})$
naturally acquires a weight as follows. 
\begin{defn}
Let $X,Y$ be metric $1$-spaces, $F,G:X\to Y$ functors, and $\alpha:F\to G$
a natural transformation. The \emph{weight }of the natural transformation
$\alpha$ is given by the formula
\[
w(\alpha)=\sup_{x\in ob(X)}\{w(\alpha_{x}:F(x)\to G(x)\}.
\]
\end{defn}
\begin{rem}
Notice that it is only the metric structure on $Y$ that is used here.
For this definition to make sense the functors $F$ and $G$ may be
any functors at all and $X$ can be an arbitrary category. However,
we will not need this extra generality. 
\end{rem}
Given metric $1$-spaces $X$ and $Y$ it is natural to consider a
subcategory of the category $\underline{\mathbf{Cat}}(X,Y)$, the
one spanned by continuous functors $F:X\to Y$. We are most interested
in the case where $X$ is object compact as well as forward and backward
compact. Corollary \ref{cor:compact domain implies for is bac} then
shows that forward and backward continuity coincide. In that case
we will simply say a functor is continuous. 
\begin{defn}
Let $X$ and $Y$ be metric $1$-spaces with $X$ object compact as
well as forward and backward compact. The \emph{mapping $1$-space}
$[X,Y]$ is the subcategory of $\underline{\mathbf{Cat}}(X,Y)$ whose
objects are the continuous functors and whose arrows are all natural
transformations between such functors. 
\end{defn}
We now establish that the weights defined above turn $[X,Y]$ into
a metric $1$-space, thus justifying its name.
\begin{lem}
\label{attained}Let $X$ and $Y$ be metric $1$-spaces with $X$
object compact as well as forward and backward compact. If $F,G:X\to Y$
are continuous and $\alpha:F\to G$ is a natural transformation \textup{then
\[
\max_{x\in ob(X)}\{w(\alpha_{x})\}
\]
exists. }\end{lem}
\begin{proof}
Let 
\[
S=\sup_{x\in ob(X)}\{w(\alpha_{x})\}
\]
which certainly exists and assume $S<\infty$. Now, let $\{x_{n}\}_{n=1}^{\infty}$
be a sequence of objects in $ob(X)$ such that $w(\alpha_{x_{n}})$
converges (in the usual sense in $\mathbb{R}$) to $S$. By object
compactness we may assume the existence of an object $x_{0}$ as in
Definition \ref{Def:obj compac}. We claim that $w(\alpha_{x_{0}})=S$,
which will prove the result. Clearly $w(\alpha_{x_{0}})\le S$, so
assume that $w(\alpha_{x_{0}})<S$ and let $\epsilon=S-w(\alpha_{x_{0}})$.
Since $F$ and $G$ are uniformly continuous find $\delta>0$ such
that for every $\psi:x\to x_{0}$ with $w(\psi)<\delta$ holds that
$w(F(\psi))<\frac{\epsilon}{3}$ and $w(G(\psi))<\frac{\epsilon}{3}$.
Now, let $n_{0}\in\mathbb{N}$ be such that for all $n>n_{0}$ holds
that $S-w(\alpha_{x_{n}})<\frac{\epsilon}{3}$. We may find an arrow
$\psi:x_{m}\to x_{0}$ such that both $w(\psi)<\delta$ and $m>n_{0}$
hold. In the corresponding naturality square
\[
\xymatrix{F(x_{m})\ar[r]^{F(\psi)}\ar[d]_{\alpha_{x_{m}}}\ar@{..>}[dr]^{\varphi} & F(x_{0})\ar[d]^{\alpha_{x_{0}}}\\
G(x_{m})\ar[r]^{G(\psi)} & G(x_{0})
}
\]
let $\varphi:F(x_{m})\to G(x_{0})$ be the common composition. Then,
by the (restricted) triangle inequality, $w(\varphi)\le w(F(\psi))+w(\alpha_{x_{0}})<\frac{\epsilon}{3}+w(\alpha_{x_{0}})=S-\frac{2}{3}\epsilon$.
However, the full triangle inequality yields that $w(\varphi)\ge w(\alpha_{x_{m}})-w(G(\psi))>S-\frac{\epsilon}{3}-\frac{\epsilon}{3}=S-\frac{2}{3}\epsilon$,
a contradiction. The proof when $S=\infty$ is similar. \end{proof}
\begin{thm}
\label{thm:Again metric}Let $X$ and $Y$ be metric $1$-spaces with
$X$ object compact as well as forward and backward compact. The mapping
$1$-space $[X,Y]$, with the weights defined above, is then a metric
$1$-space.\end{thm}
\begin{proof}
It is trivial that the identity natural transformation $id_{F}:F\to F$
has $w(id_{F})=0$ for every $F:X\to Y$ in $[X,Y]$. Now, given natural
transformations $\xymatrix{F\ar[r]^{\alpha} & G\ar[r]^{\beta} & H}
$, we need to verify that the inequalities 
\[
|w(\alpha)-w(\beta)|\le w(\beta\circ\alpha)\le w(\alpha)+w(\beta)
\]
hold . Let $x_{1},x_{2},x_{3}\in ob(X)$ be objects such that, 
\[
w(\alpha)=w(\alpha_{x_{1}}:F(x_{1})\to G(x_{1})),
\]
\[
w(\beta)=w(\beta_{x_{2}}:G(x_{2})\to H(x_{2})),
\]
and
\[
w(\beta\circ\alpha)=w(\beta_{x_{3}}\circ\alpha_{x_{3}}:F(x_{3})\to H(x_{3})).
\]
We now obtain
\[
w(\beta\circ\alpha)=w(\beta_{x_{3}}\circ\alpha_{x_{3}})\le w(\beta_{x_{3}})+w(\alpha_{x_{3}})\le w(\alpha)+w(\beta)
\]
establishing the restricted part of the triangle inequality. To obtain
the full triangle inequality we need to show that 
\[
w(\alpha)\le w(\beta\circ\alpha)+w(\beta)
\]
and that 
\[
w(\beta)\le w(\beta\circ\alpha)+w(\alpha).
\]
Indeed, 
\[
w(\beta\circ\alpha)+w(\beta)=w(\beta_{x_{3}}\circ\alpha_{x_{3}})+w(\beta)\ge w(\beta_{x_{1}}\circ\alpha_{x_{1}})+w(\beta)\ge w(\alpha_{x_{1}})-w(\beta_{x_{1}})+w(\beta)
\]
and since $w(\beta)\ge w(\beta_{x_{1}})$, by definition of $\beta$,
it follows that 
\[
w(\alpha_{x_{1}})-w(\beta_{x_{1}})+w(\beta)\ge w(\alpha_{x_{1}})=w(\alpha).
\]
The other required inequality follows similarly. 
\end{proof}
Lemma \ref{attained} and Theorem \ref{thm:Again metric} above point
already to a significant difference between our metric $1$-spaces
and the more classical Lawvere spaces, as one can easily construct
a counter example showing that mapping spaces of Lawvere spaces are
generally not Lawvere spaces, even under compactness conditions.

\subsection{The Banach fixed point theorem}

Since categories are more involved than sets, the needed condition
for the Banach fixed point result for a self map of a metric $1$-space,
presented below, requires both a categorical component and a metric
component. The metric condition of a contraction is the obvious one.
\begin{defn}
Let $F:X\to X$ be a functor from a metric $1$-space $X$ to itself.
$F$ is called a \emph{contraction }if there exists a real number
$0\le\alpha<1$ such that the inequality $w(F(\psi))\le\alpha\cdot w(\psi)$
holds for every arrow $\psi$ in $X$. 
\end{defn}
The categorical condition for a contraction is the following one.
\begin{defn}
Let $F:\mathscr{C}\to\mathscr{C}$ be a functor. A \emph{forward natural
contraction }of $F$ is a natural transformation $\alpha:id_{\mathscr{C}}\to F$
such that for every object $c\in ob(\mathscr{C})$ the equation $F(\alpha_{c})=\alpha_{Fc}$
holds. Dually, a \emph{backward natural contraction }of $F$ is a
natural transformation $\alpha:F\to id_{\mathscr{C}}$ such that for
every object $c\in ob(\mathscr{C})$ the equality $F(\alpha_{c})=\alpha_{Fc}$
holds.\end{defn}
\begin{rem}
Recall that in a category an arrow is called an \emph{epimorphism
}if it is right cancelable and a \emph{monomorphism} if it if left
cancelable. To justify the terminology in the definition above note
that a natural transformation $\alpha:id_{\mathscr{C}}\to F$ which
is an epimorphism is automatically a forward natural contraction.
Dually any natural transformation $\alpha:F\to id_{\mathscr{C}}$
which is a monomorphism is automatically a backward natural contraction.
In many categories epimorphisms correspond to surjections and monomorphisms
to injections. Thus an epimorphic $\alpha:id_{\mathscr{C}}\to F$,
being a family of epimorphisms $c\to Fc$, shows that $c$ 'surjects'
to $Fc$. Similarly, a monomorphic $\alpha:F\to id_{\mathscr{C}}$,
being a family of monomorphisms $Fc\to c$, shows that $Fc$ injects
into $c$. In both cases, $F$ is, in some sense, contracting. 
\end{rem}
Clearly, when viewing an ordinary metric space $S$ as a metric $1$-space
the notion of contraction just defined extends the classical notion.
Moreover, any function $f:S\to S$ gives rise to a corresponding functor
$F:I_{S}\to I_{S}$ which always admits a unique forward natural contraction
and a unique backward natural contraction. 
\begin{defn}
Let $F:\mathscr{C}\to\mathscr{C}$ be a functor and $\alpha:id_{\mathscr{C}}\to F$
a forward natural contraction. An arrow $\psi:c\to d$ is an $\alpha$-\emph{fixed
arrow }if $F(d)=d$ and the diagram
\[
\xymatrix{c\ar[rr]^{\alpha_{c}}\ar[dr]_{\psi} &  & Fc\ar[dl]^{F\psi}\\
 & F(d)=d
}
\]
commutes .
\end{defn}
The dual notion is that of an $\alpha$-fixed arrow for a backward
natural contraction $\alpha$. We may now state and prove the Banach
fixed point theorem for metric $1$-spaces. For simplicity we state
it for a non-degenerate metric $1$-space.
\begin{thm}
Let $X$ be a forward (respectively backward) complete non-degenerate
metric $1$-space. If $F:X\to X$ is a forward (respectively backward)
continuous contraction and $\alpha$ is a forward (respectively backward)
natural contraction of $F$ then $F$ has an $\alpha$-fixed arrow.
In particular $F$ has a fixed object. \end{thm}
\begin{proof}
The backward case is the dual of the forward case, which is the one
we prove. Fix an object $x\in ob(X)$ and construct the forward series
$\{\psi_{n}=\alpha_{F^{n}x}\}_{n=1}^{\infty}$. Note that $F\psi_{n}=\psi_{n+1}$
holds for every $n\ge0$, which thus clearly shows that the forward
series is Cauchy. The following diagram depicts this forward series
together with a forward limiting arrow $\mu_{0}:x\to y$ and a forward
limiting cone for it.
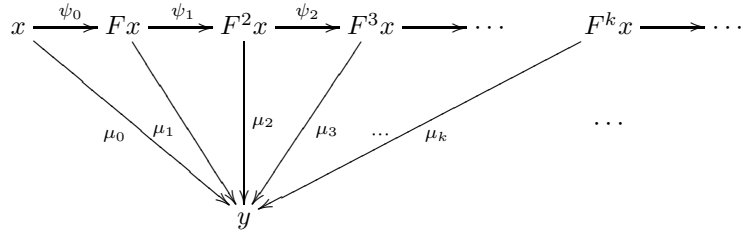
\[
\xymatrix{x\ar[r]^{\psi_{0}}\ar[ddrr]_{\mu_{0}} & Fx\ar[r]^{\psi_{1}}\ar[ddr]_{\mu_{1}} & F^{2}x\ar[r]^{\psi_{2}}\ar[dd]^{\mu_{2}} & F^{3}x\ar[ddl]^{\mu_{3}}\ar[r] & \cdots\ar@{}[ddll]^{\cdots} & F^{k}x\ar[ddlll]^{\mu_{k}}\ar[r] & \cdots\\
 &  &  &  &  & \cdots\\
 &  & y
}
\]
Applying $F$ to the diagram above yields
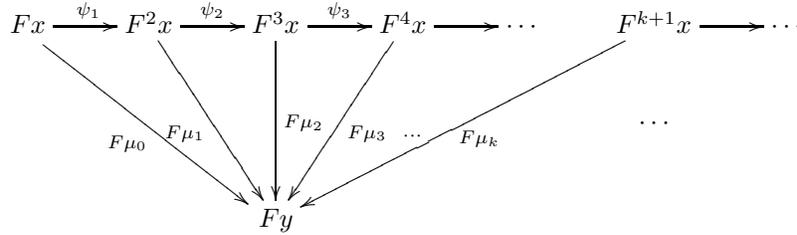
\[
\:\xymatrix{Fx\ar[r]^{\psi_{1}}\ar@<-2pt>[ddrr]_{F\mu_{0}} & F^{2}x\ar[r]^{\psi_{2}}\ar[ddr]_{F\mu_{1}} & F^{3}x\ar[r]^{\psi_{3}}\ar[dd]^{F\mu_{2}} & F^{4}x\ar[ddl]^{F\mu_{3}}\ar[r] & \cdots\ar@{}[ddll]^{\cdots} & F^{k+1}x\ar[ddlll]^{F\mu_{k}}\ar[r] & \cdots\\
 &  &  &  &  & \cdots\\
 &  & Fy
}
\]
which, since $F$ is forward continuous, is a convergent forward cone.
We may now form the following diagram
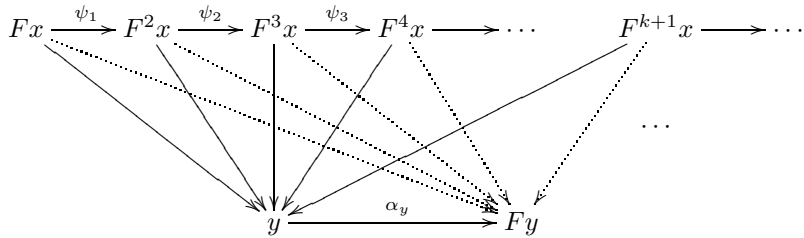
\[
\xymatrix{Fx\ar[r]^{\psi_{1}}\ar[ddrr]\ar@{..>}[ddrrrr] & F^{2}x\ar[r]^{\psi_{2}}\ar[ddr]\ar@{..>}[ddrrr] & F^{3}x\ar[r]^{\psi_{3}}\ar[dd]\ar@{..>}[ddrr] & F^{4}x\ar[ddl]\ar[r]\ar@{..>}[ddr] & \cdots & F^{k+1}x\ar[ddlll]\ar[r]\ar@{..>}[ddl] & \cdots\\
 &  &  &  &  & \cdots\\
 &  & y\ar[rr]^{\alpha_{y}} &  & Fy
}
\]
where the solid arrows are the arrows $\{\mu_{k}\}_{k=1}^{\infty}$
and the dashed arrows are $\{F\mu_{k}\}_{k=0}^{\infty}$, which we
claim commutes. To establish that, we need to show that for each $k\ge1$
the triangle
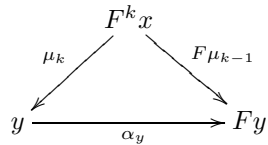
\[
\xymatrix{ & F^{k}x\ar[dl]_{\mu_{k}}\ar[dr]^{F\mu_{k-1}}\\
y\ar[rr]_{\alpha_{y}} &  & Fy
}
\]
commutes. Since $\mu_{k-1}=\mu_{k}\circ\psi_{k-1}$ and since $F\psi_{k-1}=\psi_{k}=\alpha_{F^{k}x}$,
the triangle becomes
\[
\xymatrix{F^{k}x\ar[d]_{\mu_{k}}\ar[r]^{\alpha_{F^{k}x}} & F^{k+1}x\ar[d]^{F\mu_{k}}\\
y\ar[r]_{\alpha_{y}} & Fy
}
\]
which indeed commutes since $\alpha$ is a natural transformation.
We thus established, using Lemma \ref{lem:truncations of series},
that the arrow $\alpha_{y}:y\to Fy$ is a mediating arrow for two
forward limiting cones for the same forward series and thus conclude,
by Lemma \ref{lem:mediating is 0 for series}, that $w(\alpha_{y})=0$.
Non-degeneracy of $X$ now implies that $\alpha_{y}=id$. We may thus
conclude that the first diagram above can be rewritten as
\[
\xymatrix{x\ar[r]^{\psi_{0}}\ar@<-3.5pt>[ddrr]_{\mu_{0}} & Fx\ar[r]^{\psi_{1}}\ar[ddr]_{F\mu_{0}} & F^{2}x\ar[r]^{\psi_{2}}\ar[dd]^{F^{2}\mu_{0}} & F^{3}x\ar[ddl]^{F^{3}\mu_{0}}\ar[r] & \cdots\ar@{}[ddll]^{\cdots} & F^{k}x\ar[ddlll]^{F^{k}\mu_{0}}\ar[r] & \cdots\\
 &  &  &  &  & \cdots\\
 &  & y
}
\]
and in particular the left most triangle exhibits $\mu_{0}$ as an
$\alpha$-fixed arrow. 
\end{proof}

\section{\label{sec:Dagger-structures-and}Dagger structures and a hierarchy
of symmetry}

We now turn to consider metric $1$-spaces with extra structure, namely
a dagger structure. We consider three conditions a dagger can satisfy
and identify a four level hierarchy of symmetry for metric $1$-spaces.
Interestingly, the canonical embedding of ordinary metric spaces as
metric $1$-spaces with indiscrete underlying categories takes values
in the lowest, most symmetric, level which shows again the inevitability
of symmetry discussed in Section \ref{sec:Reformulating-the-classical}. 

Recall, that a dagger category is a category equipped with an involution.
More precisely:
\begin{defn}
A \emph{dagger }structure on a category $\mathscr{C}$ is a functor
$\dagger:\mathscr{C}\to\mathscr{C}^{op}$ which is the identity on
objects and (writing $f^{\dagger}$ instead of $\dagger(f)$) so that
$f^{\dagger\dagger}=f$ holds for every arrow $f$ in $\mathscr{C}$.
A pair $(\mathscr{C},\dagger)$, with $\dagger$ a dagger structure
on $\mathscr{C}$, is called a \emph{dagger category. }\end{defn}
\begin{example}
This example will allow us later to interpret the Gromov-Hausdorff
distance as a suitable weight function $w$ on the category of cospans
of a category of metric spaces. If a category $\mathscr{C}$ admits
pushouts then one can construct its category of cospans. A cospan
in $\mathscr{C}$ is a diagram of the form
\[
\xymatrix{c\ar[dr]_{f} &  & c'\ar[dl]^{g}\\
 & d
}
\]
and can be thought of as a generalized arrow from $c$ to $c'$. Given
another cospan from $c'$ to $c''$ one can construct a cospan from
$c$ to $c''$ by considering the diagram
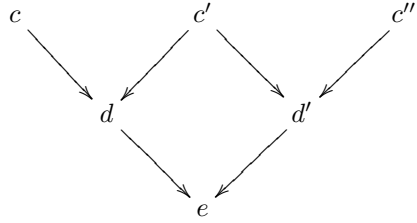
\[
\xymatrix{c\ar[dr] &  & c'\ar[dl]\ar[dr] &  & c''\ar[dl]\\
 & d\ar[dr] &  & d'\ar[dl]\\
 &  & e
}
\]
where the bottom diamond is a pushout. The generally arbitrary choice
of a particular object $e$ for the pushout implies that this composition
will only be associative up to an isomorphism. Thus, to obtain an
honest category one needs to consider the evident equivalence classes
of cospans. Once this is done one obtains the category $coSpan(\mathscr{C})$
whose objects are the objects of $\mathscr{C}$ and an arrow $c\to c'$
is an equivalence class $[(f,g)]$ of cospan (another possibility
is to construct a weak $2$-category of cospans). Note that $coSpan(\mathscr{C})$
has an obvious dagger structure sending $[(f,g)]$ to $[(g,f)]$.\end{example}
\begin{rem}
Dagger categories with some extra structure (compact closed) are studied
in \cite{abramsky2004categorical} (called there strongly compact
closed categories) to allow for an abstract categorical presentation
of quantum computations.
\end{rem}
A \emph{groupoid }is a category $\mathscr{C}$ in which all arrows
are isomorphisms. A groupoid is the many objects version of a group
in the sense that a groupoid with just one object is essentially a
group. Any groupoid has a canonical dagger structure since one can
always define $\dagger:\mathcal{G}\to\mathcal{G}^{op}$ to be the
identity on objects and to send an arrow $f:a\to b$ in $\mathcal{G}$
to its (by assumption existing and necessarily unique) inverse $f^{-1}:b\to a$. 

If $X$ is a metric $1$-space then $X^{op}$ is obviously a metric
$ $1-space by defining $w(\psi^{op}:y\to x)=w(\psi:x\to y)$ for
every $\psi^{op}$ in $X^{op}$. The following definition distinguishes
different ways that a dagger structure on $X$ can be compatible with
this metric. 
\begin{defn}
Let $X$ be a metric $1$-space. A dagger structure $\dagger:X\to X^{op}$
is called \end{defn}
\begin{itemize}
\item an \emph{iso }dagger\emph{ }structure if $w(\psi^{\dagger})=w(\psi)$
holds for every arrow $\psi$ in $X$.
\item a \emph{uniform }dagger structure if the functor $\dagger:X\to X^{op}$
is uniformly continuous.
\item a \emph{continuous }dagger structure if the functor $\dagger:X\to X^{op}$
is both forward and backward continuous.
\end{itemize}
A few comments are in order. The properties listed above are clearly
written in decreasing strength. Note, that the apparent weaker property
of a dagger structure satisfying $w(\psi^{\dagger})\le w(\psi)$ is
in fact equivalent to the iso dagger property. Note as well that each
of these properties is self dual in the sense that if $\dagger:X\to X^{op}$
is iso, uniform, or continuous then so is $\dagger^{op}:X^{op}\to X$.
Lastly, it is possible to introduce another property in between iso
and uniform, namely that $\dagger$ satisfies a Lipschitz condition. 
\begin{prop}
If $X$ is a metric $1$-space with underlying groupoid then the canonical
dagger structure on $X$ is an iso dagger structure. \end{prop}
\begin{proof}
This is just Corollary \ref{cor:wf is wfinv}.
\end{proof}
Corollary \ref{cor:if ind then symmetric} may now be restated as
simply saying that an indiscrete category is trivially a groupoid. 
\begin{prop}
\label{prop:iso dagger implies Lawvere is symm}If $X$ is a metric
$1$-space that admits an iso dagger structure then the Lawvere space
$L(X)$ is a metric space. \end{prop}
\begin{proof}
Immediate from the definition of $L(X)$. 
\end{proof}
The following lemma shows that the presence of a continuous dagger
structure implies that forward and backward notions coincide. We leave
the proof to the reader. 
\begin{lem}
\label{lem:lim colim coinc}Let $X$ and $Y$ be metric $1$-spaces,
each with a continuous dagger structure. The following then hold:\end{lem}
\begin{itemize}
\item A functor $F:X\to Y$ is forward continuous if, and only if, it is
backward continuous.
\item $X$ is forward compact if, and only if, it is backward compact.
\item $X$ is forward complete if, and only if, it is backward complete. 
\end{itemize}
Of course, these results can be refined to consider continuity at
an arrow as well as other notions of compactness and completeness. 

We can thus identify four types of symmetry a metric $1$-space may
posses. If the underlying category is a groupoid then the canonical
dagger structure is automatically an iso dagger structure and we call
such metric $1$-spaces \emph{groupoidal}. Then there are the iso
dagger structures that are not necessarily canonical, followed by
uniform dagger structures and lastly continuous ones. Metric structures
in any of these classes exhibit various degrees of unification of
the forward and backward topological notions and are thus, in a sense,
symmetric. 
\begin{rem}
\label{Rem: Michael}Referring back to Remark \ref{Rem:As-Michael-Lockyer}
we note that a generalization, along the lines carried out above,
of the axiomatization of metric spaces using the strong triangle inequality
would lead to the concept of an iso dagger metric $1$-space. It is
in this sense that we justify our claim that the approach presented
here subsumes the putative generalization of the strong triangle inequality. 
\end{rem}

\section{\label{sec:Examples,-related-structures,}Examples, related structures,
and notes}

In this final section we consider examples of metric $1$-spaces,
related structures, and we point out some of the potential applications
of metric $1$-spaces.

\subsection{Lipschitz distance}

Referring back to Example \ref{Exam:BiLip} above, clearly, every
bi-Lipschitz function is invertible and the inverse function is again
bi-Lipschitz. Thus the underlying category of the metric $1$-space
$\mathbf{BiLip}$ is a groupoid and thus the induced dagger is automatically
an iso dagger. The metric $1$-space $\mathbf{BiLip}$ is thus an
example of a groupoidal metric $1$-space and by Proposition \ref{prop:iso dagger implies Lawvere is symm},
the associated Lawvere space $L(\mathbf{BiLip})$ is a metric space.
The distance in $L(\mathbf{BiLip})$ is identical to the classical
Lipschitz distance.

\subsection{Gromov-Hausdorff distance}

Given a subset $Y$ of a metric space $X$ and $r>0$, let $B_{r}(Y)=\bigcup_{y\in Y}B_{r}(y)$.
Recall that $d_{H}(C,D)$, the Hausdorff distance between subsets
$C,D$ of a metric space $X$, is defined as the infimum over $r>0$
such that both $C\subseteq B_{r}(D)$ and $D\subseteq B_{r}(C)$.
Then, $d_{GH}(X,Y)$, the Gromov-Hausdorff distance between metric
spaces $X$ and $Y$, is defined as the infimum of $d_{H}(i(X),j(Y))$
as $i:X\to Z$ and $j:Y\to Z$ range over all possible isometric embeddings
of $X$ and $Y$ into arbitrary metric spaces $Z$. Ignoring set theoretic
difficulties, one can construct the space $\mathbb{S}$ of all metric
spaces and it is well-known that the Gromov-Hausdorff distance turns
it into a metric space. 

Consider now the category $\mathbf{Met}_{e}$ of all ordinary metric
spaces as objects and isometric embedding as arrows. Note that the
definition of the Gromov-Hausdorff distance $d_{GH}(X,Y)$ can be
restated as being the infimum over all cospans $(i,j)$ 
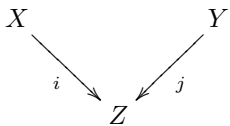
\[
\xymatrix{X\ar[rd]_{i} &  & Y\ar[ld]^{j}\\
 & Z
}
\]
from $X$ to $Y$ of a the quantity $w(i,j)=d_{H}(i(X),j(Y))$. In
fact, it is not hard to show that this definition of $w$ turns $CoSpan(\mathbf{Met}_{e})$
into a metric $1$-space and that the evident dagger structure on
the cospan category is an iso dagger. It then follows immediately
that we obtain the formula $\mathbb{S}=L(CoSpan(\mathbf{Met}_{e}))$
exhibiting the associated Lawvere space as the space of metric spaces. 
\begin{rem}
In \cite{MR2610438}, the classical Gromov-Hausdorff distance is studied
in the context of quantale enrichment. We point out that it is possible
to define quantale valued metric $1$-spaces and we suspect that the
categorical insights obtained in \cite{MR2610438} would be valid
in that setting too. 
\end{rem}

\subsection{Directed spaces}

Recall from Section \ref{sec:Dagger-structures-and} that a groupoidal
metric $1$-space carries a canonical iso dagger structure and thus
its associated Lawvere space is an ordinary metric space. We propose
that groupoidal metric $1$-spaces $X$ can naturally serve as models
for directed spaces. In any groupoidal metric $1$-space $X$ each
set of arrows $X(x,y)$ is a heap and if we assume $X$ is a connected
groupoid (i.e., every two objects are connected by some arrow) then
all such heaps are isomorphic, let's say to $H$. It can be shown
that the smaller $H$ is the more symmetric the weight function must
be. In more detail, let $\sigma(x,y)=\sup|w(\psi)-w(\varphi)|$ where
$\psi:x\to y$ and $\varphi:y\to x$ range over all possible such
arrows. Then the smaller the size of $H$ is with respect to the cardinality
of $ob(X)$ (we assume now that it is a set) the smaller the upper
bound of $\sigma(x,y)$ is. The extreme case where $|H|=1$ forces
$\sigma(x,y)=0$ and thus the weight function is completely symmetric.
The case where $ob(X)=2$ and $|H|=2$ can easily be analyzed to produce
an exact upper bound for $\sigma(x,y)$, which is different than $0$.
Further research is required to produce workable models of, for instance,
the directed real line modeled by a groupoidal metric $1$-space whose
underlying groupoid has the real numbers as objects and, between any
two real numbers, a continuum of arrows. 
\begin{rem}
We wish to emphasize a conceptual difference between the directed
extension of metric spaces we present here and Grandis' extension
of topology to directed topology (see \cite{MR2562859}). Grandis
defines a $d$-space to be a topological space together with a set
$D$ of \emph{allowed }paths in $D$. Mappings between $d$-spaces
are then required to respect these chosen paths. Thus, underlying
a directed space is an undirected space. In other words, directedness
is more structure. In contrast, in our approach symmetry is seen as
extra structure that might be present on an otherwise non-symmetric
structure. In other words, undirectedness is a property. 
\end{rem}

\subsection{Bi-metric spaces and phase space of interfering measurements}

Consider a metric $1$-space $X$ such that $X(x,y)$, for every two
objects $x,y\in ob(X)$, consists of precisely two arrows which we
denote by $\pm1_{xy}:x\to y$, with composition given by multiplication.
Let us assume that for every $x\in ob(X)$ holds that $w(-1_{xx})=h$,
a constant. If we interpret, for $x\ne y$, $a_{1}=w(1_{xy})$ as
the accuracy in measuring quantity $q_{1}$ from $x$ to $y$ and
$a_{2}=w(-1_{xy})$ as the accuracy of measuring quantity $q_{2}$
from $x$ to $y$ then the full triangle inequality reads
\[
|a_{1}-a_{2}|\le h\le a_{1}+a_{2}
\]
and thus the constant $h$ can be seen as setting lower and upper
bounds the accuracy in simultaneous measurements of different quantities.
The upper bound implies that the quantities measured are related and
the upper bound implies that the measurements interfere. 

The structure described above is certainly related to a set equipped
with two metric structures. Such structures originated with Rosen's
\cite{MR0406361} which gave rise to what is called bi-metric theories
of gravitation. More recent ideas include \cite{MR2802480} on spacetime
as a pseudo-Finslerian bi-metric space and \cite{MR2771706} on quantum
Einstein gravity. Further research is needed to determine the applicability
of metric $1$-spaces to modeling phase space of quantum measurements
and their relation to bi-metric spaces.

\subsection{Metric $n$-spaces }

The passage from metric structures based on sets to metric structures
based on categories naturally raises the question as to metric structures
based on higher dimensional categories. Very loosely speaking, an
$n$-dimensional category has various cells of dimensions $k$ for
$0\le k\le n$ such that each cell has a boundary composed of cells
of lower dimensions 'glued' together. On top of this structure there
is a composition rule that dictates how cells with matching boundaries
can be composed. This composition is required to satisfy some notion
of associativity and typically comes in two flavours: strict and weak.
The combinatorial freedom allowed by higher dimensional cells makes
turning this idea into a rigorous definition highly non-trivial. Indeed,
especially for the weak flavour of higher categories, there are more
than a dozen different definitions proposed and for most of them it
is not yet known if they are equivalent or not. A survey of such definitions
can be found in \cite{MR1883478}. In any case, defining metric $n$-spaces
carries with it all of the difficulties present in the theory of higher
categories and thus we adopt a very informal approach to explain the
following construction. 

Assume that some notion of metric $n$-space is given where such an
$n$-space is an $n$-category in which each cell has a weight together
with some compatibilities. We describe, very non-rigorously, a recursive
construction that, given a metric $n$-space, produces what we call
its associated \emph{$n$-simplex metric.} Thus, let $X$ be a metric
$n$-space and consider $n+1$ objects $x_{0},\cdots,x_{n}$ in $ob(X)$.
If $n=1$ then we set 
\[
d_{1}(x_{0},x_{1})=\inf_{\psi:x_{0}\to x_{1}}\{w(\psi)\}.
\]
This is the same formula defining the associated Lawvere space $L(X)$.
We now call it the associated $1$-simplex metric on $ob(X)$. Assume
that for $1\le k\le n$ we have defined the notion of the $k$-simplex
metric $d_{k}(x_{0},\cdots,x_{k})$ on $ob(X)$ associated to a metric
$k$-space $X$. We now define the $(n+1)$-simplex metric associated
to a metric $(n+1)$-space $X$. Given points $x_{0},\cdots,x_{n}\in ob(X)$
consider for any $n$ of these points $x_{0},\cdots,\hat{x_{j}},\cdots,x_{n}$
the value of $d_{n}(x_{0},\cdots,\hat{x_{j}},\cdots,x_{n})$. Let
us assume for simplicity that it is obtained at a unique $n$-cell
of $X$ which we denote by $\alpha_{j}$. Now, define 
\[
d_{n+1}(x_{0},\cdots,x_{n})=\inf_{\beta}\{w(\beta)\}
\]
where $\beta$ ranges over all $(n+1$)-cells having as boundary the
composition of $\alpha_{0},\cdots,\alpha_{n}$. 

As stated, this is not a rigorous definition and further research
is needed to make it precise and study its properties. We can make
a few remarks we believe to be important since they relate this construction
to existing notions. 

The case $n=1$ clearly gives back the construction $L(X)$ of Lawvere
space associated to a metric $1$-space. The case $n=2$ produces
a notion of metric structure which is related to structures called
$2$-metric spaces studied, for instance, in \cite{MR2254125} and
\cite{aliouche2010fixed} with the latter source containing many references
to other sources. In general, what we call the $n$-simplex metric
is related to a notion called $n$-hemi-metric in \cite{deza2009encyclopedia}.

We conclude this final section by responding to the following question
posed, at the end of the introduction, in \cite{aliouche2010fixed}
by Aliouche and Simpson. In light of the relation between metric spaces
and categories enriched in $\mathbb{R}_{\ge0}$ (namely, that Lawvere
spaces, seen as generalized metric spaces, are precisely categories
enriched in $\mathbb{R}_{\ge0}$ with a suitable monoidal structure)
what is the category theoretical counterpart of $2$-metric spaces
(as they define in their article). In light of the work above we propose
the following.

Consider the category $\mathbf{Set}_{w}$ of weighted sets. A typical
object in it is a set $S$ together with a function $S\to\mathbb{R}_{+}$
assigning a weight $s_{w}$ to each element $s\in S$. A typical arrow
$f:S\to S'$ is then a function such that $f(s)_{w}\le s_{w}$ for
all $s\in S$. A category enriched in $\mathbf{Set}_{w}$ amounts
to a category with a weight for each arrow such that each identity
arrow has weight $0$ and the restricted triangle inequality holds.
Thus, our notion of metric $1$-spaces (with the more restricted class
of non-expanding functors) can be identified with a proper subcategory
of the category of categories enriched in $\mathbf{Set}_{w}$. This
point of view extends the fact that Lawvere spaces are equivalent
to categories enriched in $\mathbb{R}_{+}$ and ordinary metric spaces
are equivalent to a proper subcategory thereof. 

Similarly, one may consider $2$-categories weakly enriched in $\mathbf{Set}_{w}$
where the notion of metric $2$-spaces (with non-expanding functors)
can again be identified with a proper subcategory of such enriched
$2$-categories. Just as Lawvere spaces arise as the $1$-simplex
metric associated to a metric $1$-space we hypothesize that $n$-categories
weakly enriched in $\mathbb{R}_{+}$ arise as the $n$-simplex metric
associated to metric $n$-spaces. 

\bibliographystyle{plain}
\bibliography{references}

\end{document}